\documentclass[12pt]{article}
\usepackage{lipsum}
\usepackage{fancyhdr}
\usepackage{booktabs}
\usepackage{array}
\usepackage{amssymb}
\usepackage{parskip}
\usepackage{hyperref}
\usepackage{enumitem}
\usepackage{mathtools} 
\usepackage{xparse,textcomp}
\DeclarePairedDelimiter\floor{\lfloor}{\rfloor}

\usepackage[lofdepth,lotdepth]{subfig}

\usepackage{cite}

\usepackage{parskip}
\usepackage[margin=0.8in]{geometry}
\newcommand{\forceindent}{\leavevmode{\parindent=1.5em\indent}}
\usepackage{algorithm} % Algorithms environment
\usepackage{algpseudocode}
\newcommand{\R}{\mathbb R}

\newcommand{\N}{\mathbb N}

\newcommand{\minmath}{\mathop{}\!\mathrm{min}}

\usepackage{amsmath,amsthm}
\makeatletter

\usepackage[belowskip=-30pt,aboveskip=10pt]{caption} % spacing before and after caption
\usepackage{amsfonts}
\usepackage{amsopn}

\usepackage{amssymb,amsmath}

\renewcommand\section{\leftskip 0pt\@startsection {section}{1}{\z@}%
	{-3.5ex \@plus -1ex \@minus -.2ex}%
	{2.3ex \@plus.2ex}%
	{\normalfont\Large\bfseries}}

\renewcommand\subsection{\leftskip 0pt\@startsection{subsection}{2}{\z@}%
	{-3.25ex\@plus -1ex \@minus -.2ex}%
	{1.5ex \@plus.2ex}%
	{\normalfont\large\bfseries}}

\renewcommand\subsubsection{\leftskip 0pt\@startsection{subsubsection}{3}{\z@}%
	{-3.25ex\@plus -1ex \@minus -.2ex}%
	{1.5ex \@plus .2ex}%
	{\normalfont\large\bfseries}}
\usepackage{graphicx}

\usepackage{geometry}
\newtheorem{definition}{Definition}

\newtheorem{Theorem}{Theorem}[section]

\theoremstyle{remark}
\newtheorem{remark}[Theorem]{Remark}

\begin{document}
	
\title{ \centering{Approximation of Functions on Manifolds in High Dimension from Noisy Scattered Data}}
\author{Shira Faigenbaum-Golovin${^{1, *}}$~~David Levin${^1}$ \\
	\small{${^1}$ School of Mathematical Sciences, Tel Aviv University, Israel}
	\\
	\small{${^*}$ Corresponding author, E-mail address: alecsan1@post.tau.ac.il} %, \email{alecsan1@post.tau.ac.il}
}


\maketitle
\begin{abstract}
In this paper, we consider the fundamental problem of approximation of functions on a low-dimensional manifold embedded in a high-dimensional space, with noise affecting both in the data and values of the functions. Due to the curse of dimensionality, as well as to the presence of noise, the classical  approximation methods applicable in low dimensions are less effective in the high-dimensional case.  We propose a new approximation method that leverages the advantages of the Manifold Locally Optimal Projection (MLOP) method  \cite{faigenbaumgolovin2020manifold} and the strengths of the method of Radial Basis Functions (RBF) \cite{dyn1983iterative}. The method is parametrization free,  requires no  knowledge regarding the manifold's intrinsic dimension, can handle noise and outliers in both the function values and in the location of the data, and is applied directly in the high dimensions. We show that the complexity of the method is linear in the dimension of the manifold  and squared-logarithmic in the dimension of  the codomain of the function. Subsequently, we demonstrate the effectiveness of our approach by considering different manifold topologies and show the robustness of the method to various noise levels.

\end{abstract}

\noindent\textbf{keywords:} Manifold learning, Approximation of functions, High dimensions, Dimensional reduction, Noisy data 

\noindent\textbf{MSC classification:} 65D99 \\
(Numerical analysis - Numerical approximation and computational geometry)

\section{Introduction}
In this paper, we consider the following formulation of the problem of approximation of functions in a high-dimensional space. Let $f: \mathbb{R}^n \to \mathbb{R}^s$ be a function, and let $\{f(x_i)\}_{i=1}^K$ be its values on a given sample set of points $\{x_i\}_{i=1}^K \subset \mathbb{R}^n$ with noise present both in the domain of the function and in its codomain. The goal of the approximation is to estimate the values of the function at a new set of points. While in low dimensions numerous methods were suggested to solve this problem (e.g., splines, or Moving Least-Squares \cite{levin1998approximation}), in high dimensions this is a challenging task due to the presence of noise and the curse of dimensionality. For instance, with respect to the latter challenge, if one merely assumes that the function is smooth, then approximation rates deteriorate severely with the growth of the dimension, the reason being that the amount of sampled data should grow exponentially with respect to the dimension if one wishes to maintain the same order of approximation. 

\forceindent We categorize high-dimensional approximation methods according to whether the  domain of the function to be approximated is a manifold or not. If no assumptions are made on the data domain, several methods were suggested. For example, solutions which treat non-smooth multivariate functions, \cite{amir2018high}, are based on sparse occupancy trees \cite{binev2011fast}, Radial Basis Functions \cite{dyn1983iterative} (which we will discuss in detail below), or address the problem in the case where the values of the function lie on a manifold \cite{grohs2013projection}. 

\forceindent In many situations, the high-dimensional data reside on a low-dimensional manifold, and this information can be exploited to improve the approximation via one of  the following two approaches: approximating in low dimension after dimension reduction, or alternatively approximating in high dimension. At times, reducing the dimension (e.g., in PCA \cite{pearson1901liii}, Multidimensional Scaling \cite{cox2000multidimensional}, Linear Discriminant Analysis \cite{fisher1936use}, Locality Preserving Projections \cite{he2004locality},  Locally Linear Embedding \cite{roweis2000nonlinear}, ISOMAP \cite{tenenbaum2000global}, Diffusion Maps \cite{coifman2005geometric}, and Neural Networks in their general form, \cite{lin2008riemannian}) can lead to a better approximation (in terms of handling the challenge of the dimensionality, as well as the noise in the data). However, it may be non-efficient if the data volume is very large, and in addition, may result in information loss (due to some assumptions that need to be made on the data, e.g., regarding the data geometry, or the intrinsic dimension).

\forceindent On the other hand, the assumption that the data reside on a manifold can be utilized  to improve the approximation in high dimensions. Approximation of functions on manifolds is studied using local polynomials \cite{bickel2007local}, wavelets \cite{coifman2006diffusion}, local linear regression \cite{bickel2007local}, or neural networks \cite{andras2017high, shaham2018provable, chen2019efficient}. For smooth functions on $[0, 1]^N$ which depend on a much smaller number $l$  of variables, a solution was suggested in \cite{devore2011approximation}. In addition, a recent paper, \cite{sober2017approximation},  proposed a solution based on Moving Least-Squares (MLS) that was designed to deal with noisy data with good rates of approximation.

\forceindent In this paper, we propose a method of  approximation of functions that leverages the advantages of the Manifold Locally Optimal Projection (MLOP) algorithm \cite{faigenbaumgolovin2020manifold} to complement the strengths of the method of  Radial Basis Functions (RBF) \cite{dyn1983iterative, buhmann2003radial}. We introduce this duet for approximation in high dimensions under noisy conditions (both in the domain and in the codomain of the function). In what follows we will provide a short introduction to the  RBF  method as well as to the Locally Weighted Average Approximation method (which will be used to contrast the RBF numerical results). In the next section, we will explain the proposed methodology, and demonstrate the improvement in approximation via several numerical examples.

\section{Preliminaries}
\subsection{Preliminaries --- the MLOP framework}
\label{MLOP_pre}	
%Extending The Manifold-MLS to Function Approximation
The Locally Optimal Projection (LOP) method  was introduced in \cite{lipman2007parameterization} to approximate two-dimensional surfaces in $\mathbb{R}^3$ from point-set data. The procedure does not require the estimation of local normals  and planes,  or parametric representations. Its main advantage  is that it performs well in the case of noisy samples. In \cite{faigenbaumgolovin2020manifold} the LOP mechanism was generalized to devise the so-called \textit{Manifold Locally Optimal Projection} (MLOP) method. Here, we give a concise overview of the MLOP method and its key properties.

\forceindent First, we introduce the $h$-$\rho$ condition, defined for scattered-data approximation of functions (which is an adaptation of the condition in \cite{levin1998approximation} for low-dimensional data), to handle finite discrete data on manifolds.

\begin{definition}\label{def:def4}
	\textbf{$h$-$\rho$ sets of fill distance $h$  and density $\leq \rho$} with respect to a manifold $\mathcal{M}$. Let $\mathcal{M}$ be a manifold in $\mathbb{R}^n$ and consider sets of data points sampled from $\mathcal{M}$. We say that such a set $P=\{{P_j }\}_{i=1}^J$ is an 
	$h$-$\rho$ set if:\\
	1.   $h$ is the fill distance, i.e., $h=\text{median}_{p_j \in P} \minmath_{p_j \in P \backslash\{p_i\}} \|p_i-p_j\|$. \\
	2.   $\#\{P \cap \bar{B}(y,kh)\}\leq \rho k^n, \quad k \geq 1, \quad y\in \mathbb{R}^n$. \\
	Here $\#Y$ denotes the number of elements in a set $Y$ and $\bar{B}(x,r)$ denotes the closed ball of radius $r$ centered at $x$. 
\end{definition}
\vspace{-6mm}
\forceindent Note that the last condition regarding the point  separation $\delta$ defined in \cite{levin1998approximation}, which states that there exists $\delta >0$ such that  $\|p_i-p_j \| \geq \delta, \quad 1\leq i \leq j \leq J$, is redundant in the case of finite data. We also note   that the vanilla definition of the  fill distance uses the supremum $\sup$ in its expression;  here we use the median in order to deal with the presence of outliers.

\forceindent The setting of the high-dimensional reconstruction problem is the following: Let $\mathcal{M}$ be a manifold  in $\mathbb{R}^n$  of unknown intrinsic dimension $d \ll n$. There  is given a noisy point-cloud $P=\{p_j\}_{j=1}^J \subset \mathbb{R}^n$ situated near  the manifold $\mathcal{M}$   such that   $P$ is a $h$-$\rho$ set. We wish to find a new point-set $Q=\{q_i\}_{i=1}^I \subset \mathbb{R}^n$  which will serve as a noise-free approximation of $\mathcal{M}$. We seek a solution in the form of a new, quasi-uniformly distributed point-set $Q$ that will replace the given data $P$  and provide a noise-free approximation of $\mathcal{M}$. This is achieved by leveraging the well-studied weighted $L_1$-median \cite{vardi2000multivariate} used in the LOP algorithm and requiring a quasi-uniform distribution of points $q_i \in Q$. 
These ideas are encoded by the cost function
\begin{equation} 
\label{eq:1}	 G(Q) = E_1(P,Q)+\Lambda E_2(Q) = \sum\limits_{q_i \in Q} \sum\limits_{p_j \in P} \|q_i-p_j\|_{H_{\epsilon}} w_{i,j} +
\sum\limits_{q_i \in Q} \lambda_i \sum\limits_{q_{i'} \in Q \backslash \{q_i\}} \eta(\|q_i-q_i'\|) \widehat w_{i,i'}\,, 
\end{equation}
where the weights $w_{i,j}$ are rapidly decreasing smooth functions. The MLOP implementation uses  $w_{i,j} = \exp\big\{\!-\|q_i-p_j\|^2/{h_1^2}\big\}$
%$e^{-\frac{\|q_i-p_j\|^2}{h_1^2}}$
and $\widehat w_{i,i'} = \exp\big\{\!-\|q_i-q_i'\|^2/{h_2^2}\big\}$.
The  $L_1$-norm used in \cite{lipman2007parameterization}  is replaced by the ``norm" $\| \cdot \|_{H_{\epsilon}}$ introduced in \cite{levin2015between} as $\|v\|_{H_\epsilon}=\sqrt{v^2+\epsilon}$, where $\epsilon >0$ is a fixed parameter (in our case we take $\epsilon=0.1$). As shown in \cite{levin2015between}, using $\| \cdot \|_{H_{\epsilon}}$ instead of $\| \cdot \|_1$ has the advantage  that one works  with a smooth cost function  and outliers can be removed. In addition, $h_1$ and $h_2$ are the support size parameters of $w_{i,j}$ and $\widehat w_{i,i'}$ which guarantee a sufficient amount of $P$ or $Q$ points for the reconstruction (for more details, see the subsection  ``Optimal Neighborhood Selection"  in \cite{faigenbaumgolovin2020manifold}). Also, $\eta(r)$ is a decreasing function such that $\eta(0)= \infty $; in our case we take $\eta(r) = \frac {1} {3r^3}$. Finally, $\{\lambda_i\}_{i=1}^I$ are constant balancing parameters.

In order to solve the problem with the cost function \eqref{eq:1}, we look for a point-set $Q$ that minimizes $G(Q)$. 
The solution $Q$ is found via the gradient descent iterations
\begin{equation}
q_{i'}^{(k+1)}=q_{i'}^{(k)}-\gamma_k \nabla G(q_{i'}^{(k)}),\qquad i'=1,\dots,I \,, 
\end{equation}
where the initial guess $\{q_i^{(0)}\}_{i-1}^I=Q^{(0)}$ consists of points   sampled from $P$. \\
The gradient of $G$ is  
\begin{equation}  \label{eq:GradG} \nabla G(q_{i'}^{(k)}) = \sum\limits_{j=1}^J {\big(q_{i'}^{(k)}-p_j\big) \alpha_j^{i'}}- \lambda_{i'}\sum\limits_{\substack {i=1 \\ i\neq i'}}^I{ \big(q_{i'}^{(k)} - q_{i}^{(k)} \big) \beta_i^{i'}}\,, \end{equation}
with the coefficients $\alpha_j^{i'}$ and $\beta_j^{i'}$  given by the formulas
\begin{equation}
\label{eq:alpha}
\alpha_j^{i'}  = \frac{w_{i,j}}{ \|q_i-p_j\|_{H_{\epsilon}}} \left(1-\frac{2}{h_1^2}\|q_i-p_j\|_{H_{\epsilon}}^2\right)
\end{equation}
and 
\begin{equation}
\label{eq:beta}
\beta_i^{i'} = \frac{\widehat w_{i,i'}}{ \|q_i-q_{i'}\|} \left( \left| {\frac{\partial \eta \left( \|q_i-q_{i'}\| \right) } {\partial r}}\right|  +
\frac{2\eta \left( \|q_i-q_{i'}\| \right) } {h_2^2}   \|q_i-q_{i'}\|
\right),	 
\end{equation}
for 	$i=1,...,I$, $i\ne i'$.
In order to balance the two terms in $\nabla G(q_{i'}^{(k)})$, the factors $\lambda_{i'}$ are initialized in the first iteration as
\begin{equation}  \label{eq:2}
\lambda_{i'} = -\,\frac{\bigg\|\sum\limits_{j=1}^J {\big(q_{i'}^{(k)}-p_j\big) \alpha_j^{i'}} \bigg\|}{\bigg\|\sum\limits_{i=1}^I{\big(q_{i'}^{(k)} - q_{i}^{(k)} \big) \beta_i^{i'}} \bigg\|}\,.
\end{equation}
Balancing the contribution of the two terms is important in order to maintain equal influence of the attraction and repulsion forces in $G(Q)$.
The step size in the direction of the gradient $\gamma_k$ is calculated as indicated in \cite{barzilai1988two}:
\begin{eqnarray} \label{gamma_k} \gamma_k = \frac{\langle \bigtriangleup q_{i'}^{(k)},  \bigtriangleup G_{i'}^{(k)} \rangle} {\langle \bigtriangleup G_{i'}^{(k)},  \bigtriangleup G_{i'}^{(k)} \rangle}\,,  \end{eqnarray}
where  $\bigtriangleup q_{i'}^{(k)}  = q_{i'}^{(k)}  - q_{i'}^{(k-1)}$ and $\bigtriangleup G_{i'}^{(k)}  = \nabla G_{i'}^{(k)}  - \nabla G_{i'}^{(k-1)}$.\\
\vspace{-6mm}
\begin{remark}\label{def:red_dim}
	The reasoning in terms of Euclidean distances, which is the cornerstone of the MLOP method, works well in low dimensions, e.g., for the reconstruction of surfaces in 3D, but breaks down in high dimensions once  noise  is present. To deal with this issue, a dimension reduction is performed via random linear sketching \cite{woodruff2014sketching}. It should be emphasized that the dimension reduction procedure is utilized solely for the calculation of norms, and the manifold reconstruction is performed in the high-dimensional space. Thus, given a point $x \in \mathbb{R}^n$, we project it to a lower dimension $m \ll n$ using a random matrix, $S$, with certain properties. Subsequently, the norm of  $\|S^t x\|$ will approximate $\|x\|$. The construction of $S$ is carried out in the following steps:
	\begin{enumerate}[noitemsep]
		\item Sample $G\in \R^{J\times m}$ with $G\sim N(0, 1)$.
		\item Compute $B \in \R^{n \times m}$ as $B:=P^{\rm t}G$.
		\item Calculate the QR decomposition of $B$ as $B = SR$, and use $S$ as the dimension reduction matrix.
	\end{enumerate}
\end{remark}

\textbf{Preliminaries --- Optimal Neighborhood Selection}

\forceindent The support sizes $h_1$  and $h_2$ of  the functions $w_{i,j} = \exp\big\{\!-\|q_i-p_j\|^2/{h_1^2}\big\}$
%$e^{-\frac{\|q_i-p_j\|^2}{h_1^2}}$
and $\widehat w_{i,i'} = \exp\big\{\!-\|q_i-q_i'\|^2/{h_2^2}\big\}$, respectively, are closely related to the fill distance of the $P$-points and the $Q$- points. 
Due to the importance of optimal selection of these parameters, we quote here several definitions and results from \cite{faigenbaumgolovin2020manifold}. Unlike the standard definition of the notion of fill distance in scattered-data function approximation \cite{levin1998approximation}, we introduce
\begin{definition}\label{def:def1}
	The fill distance of the set $P$ is 	
	\begin{eqnarray} \label{h_0} h_0=\text{median}_{p_i\in P} \minmath_{p_j\in P \backslash \{p_i\}} \|p_i-p_j \| \,. \end{eqnarray}
	Note  that the vanilla definition of fill distance uses the supremum in the definition (instead of the \textup{median}). However, as mentioned above, in our case we replace  the supremum with  the median so as to deal with the presence of outliers.
\end{definition}

\begin{definition}\label{def:def2}
	Given two point-clouds, $P=\{p_j\}_{j=1}^J \subset \mathbb{R}^n$ and $Q=\{q_i\}_{i=1}^I \subset \mathbb{R}^n$, situated near a manifold $\mathcal{M}$ in $\mathbb{R}^n$, such that their sizes obey the constraint $I\leq J$, denote $\nu= \floor*{\frac{J} {I}}$. Then we say that the radius that guarantees approximately $\nu$ points from $P$ in the support of each point $q_i$ is $\widehat{h}_0=c_1 h_0$, with $c_1$ given by
	\begin{eqnarray} \label{c_1}  \widehat{h}_0=c_1 h_0,~{\rm with}~ ~c_1=\text{argmin} \{c: \#(\bar{B}_{ch_0}(q_i) \cap P)\geq \nu,\,  \forall q_i\in Q\}\,. \end{eqnarray}
\end{definition}

\begin{remark}\label{def:def3}
	Let $\sigma$ be the variance of the Gaussian $w(r) = \exp\{- {r^2}/{\sigma^2}\}$. For the normal distribution, four standard deviations away from the mean account for $99.99\%$ of the set.
	In our case, by the definition of $w_{i,k}$, since $h$ is the square root of the variance, $4\sigma= 4\frac {h}{\sqrt{2}}= 2 \sqrt{2}h_1$ covers $99.99\%$ of the support size of $w_{i,k}$.
\end{remark}
%https://en.wikipedia.org/wiki/68%E2%80%9395%E2%80%9399.7_rule

The following theorem, proved in \cite{faigenbaumgolovin2020manifold}, indicates how the parameters $h_1$ and $h_2$  should be selected.

\begin{Theorem}\label{lma0}
	Let $\mathcal{M}$ be a $d$-dimensional manifold in $\mathbb{R}^n$. Suppose given two point-clouds, $P=\{p_j\}_{j=1}^J \subset \mathbb{R}^n$ and $Q=\{q_i\}_{i=1}^I \subset \mathbb{R}^n$, situated near   $\mathcal{M}$, such that their sizes obey the constraint $I\leq J$, and let $\nu= \floor*{\frac{J} {I}}$. Let $w_{i,j}$ be the locally supported weight function given by $w_{i,j} = \exp\{- {\|q_i-p_j\|^2}/{h^2}\}$. Then a neighborhood size of $h = 2 \sqrt{2}\widehat{h}_0$ guarantees $ 2^{1.5d}\nu$ points in the support of $w_{i,j}$, where $\widehat{h}_0=c_1h_0$, with $c_1$ given by  \textup{\eqref{c_1}}.
\end{Theorem}

\textbf{Theoretical Analysis of the MLOP Method}

For the sake of completeness, we mention here several important results regarding the convergence of the MLOP method, its order of approximation,  rate of convergence  and complexity  (for more details, see \cite{faigenbaumgolovin2020manifold}). 

\begin{Theorem}[Convergence to a stationary point]\label{thm:conv}
	Let $\mathcal{M}$ be a  in $\mathbb{R}^n$ of unknown intrinsic dimension $d$. Suppose that the scattered data points $P = \{{p_j }\}_{j =1}^J$ were sampled near the manifold $\mathcal{M}$, $h_1$ and $h_2$ are set as  in Theorem  \textup{\ref{lma0}}, and the $h$-$\rho$ set condition is satisfied with respect to $\mathcal{M}$. Let the points $Q^{(0)}=\{q_i^{(0)} \}_{i=1}^I$ be sampled from $P$. Then the gradient descent iterations \textup{\eqref{eq:2}} converge almost surely to a local minimizer $Q^*$.
\end{Theorem}

\begin{Theorem}[Order of approximation]\label{thm:approx_order}
	Let $P=\{p_j\}_{j=1}^J$ be a set of points that are sampled \textup{(}without noise\textup{)} from a $d$–dimensional $C^{2}$ manifold $\mathcal{M}$, and satisfy the $h$-$\rho$ condition. Then for a fixed $\rho$  and a finite support of size $h$ of the weight functions  $w_{i,j}$, the set $Q$ defined by the \textup{MLOP} algorithm has an order of approximation $O(h^{2})$ to $\mathcal{M}$.
\end{Theorem}

\begin{definition}\label{def:def6}
	A differentiable function $f(\cdot)$ is called $L$-smooth if for any $x_1, x_2$
	\begin{eqnarray*}  \|\nabla	 f(x_1) - \nabla f(x_2)\| \leq L\|x_1 - x_2\| \,.\end{eqnarray*}
\end{definition}

\begin{Theorem}[Rate of convergence]\label{thm:convRate}
	Suppose the point-set $P=\{{p_j }\}_{j =1}^J$ is  sampled near a $d$-dimensional manifold in $\mathbb{R}^n$ and the assumptions of  Theorem~{\rm \ref{thm:conv}} are satisfied. Suppose the cost function $G$ defined   in {\rm (\ref{eq:1})} is an $L$-smooth function. For  $\epsilon>0$, let $Q^*$ be a local fixed-point solution of the gradient descent iterations, with step size $\gamma =  \epsilon^{-1}$. Set the termination condition as $\|\nabla f(x)\| \leq \epsilon$. Then $Q^*$ is an $\epsilon$-first-order stationary point that will be reached after $k = L(G(Q^{(0)}) – G(Q^*))\epsilon^{-2}$ iterations, where $L=l^2$ and $l < \infty $ is a bounded parameter.
\end{Theorem}

\begin{Theorem}[Complexity]
	\label{complex_MLOP}
	Given a point-set $P=\{{p_j }\}_{j =1}^J$ sampled near a $d$-dimensional manifold $\mathcal{M} \in \mathbb{R}^n$,  let $Q=\{q_i \}_{i=1}^I$ be a set of points that will provide the desired manifold reconstruction. Then the complexity of the {\rm MLOP} algorithm is $O(nmJ + k I(nm\widehat{I}+\widehat{J}))$, where the number of iterations $k$ is bounded as in Theorem~{\rm \ref{thm:convRate}}, $m \ll n$ is the smaller dimension to which we reduce the dimension of the data, and $\widehat{I}$ and $\widehat{J}$ are the numbers of points in the support of the weight functions $\widehat w_{i,i'}$, $w_{i,j}$   that belong to the $Q$-set and $P$-set, respectively. Thus, the approximation is linear in the ambient dimension $n$, and does not depend on the intrinsic dimension $d$.
\end{Theorem}

\subsection{Preliminaries --- Radial Basis Functions}
\forceindent \textbf{Radial Basis Functions} constitute a very useful and convenient   multivariate interpolation  tool \cite{dyn1983iterative, buhmann2003radial}. Given the values of a function $f: \mathbb{R}^n \to \mathbb{R}^s$ at center points $x_i \in \mathbb{R}^n$, $i=1,..., K$, we approximate the value of $f$ at a new point $x$ by the formula
\begin{equation} 
\label{RBF} \widetilde{f}(x) = \sum \limits_{i=1}^K \lambda_i \phi(\|x - x_i \|) \,, 
\end{equation}
where $\phi: \R_+ \to \R$ is a radial basis function and $\lambda_i$ are scalar parameters chosen to maintain interpolation at the center points, i.e., $\widetilde{f}(x_i) = f(x_i)$.
For  examples of possible choices of radial basis functions, see \cite{buhmann2020analysis, wu1993local}. In the numerical examples presented below we choose to use the following Gaussian RBF with local support: \[
\begin{aligned}
&\phi_1(r) = \exp\{- ({r}/{h})^2\}, \\
&\phi_2(r) = \exp\{-({r}/{h})^2\}\big(1 + r/h\big),\\
&\phi_3(r) = \exp\{-({r}/{h})^2\}\big(15 + 15(r/h) + 6 (r/h)^2 + (r/h)^3\big).
\end{aligned}
\]

\forceindent In what follows, we will state a theorem proved in \cite{yoon2001spectral} on the order of approximation in a Sobolev space of a method that uses a  RBF of general form, namely
\begin{equation}
\label{RBF2} \widetilde{f}(x) = \sum \limits_{j=1}^K \lambda_j \phi_w(\|x - x_j \|) + \sum \limits_{i=1}^I \alpha_i p_i(x)\,, \end{equation}
where $\phi_w := \phi(\cdot/w)$, with $w$ depending on the fill-distance $h$ of the points $X = \{x_j\}_{j=1}^K$, $p_1, ..., p_I$ is a polynomial basis for $\Pi_m$ (the space of polynomials of degree $\le m$), and the coefficients $\lambda_j$ and $\alpha_j$ are chosen to satisfy the linear system $\widehat f(x_j) = f(x_j)$ for $j=1,...,K$
and $\sum \limits_{j=1}^K \lambda_j p_i(x_j) = 0$, $i=1,...,I$.

To state the  theorem on the order of approximation  of the RBF method we nee some additional definitions. 
\begin{definition}\label{def:def13}
	For  $k \in \N$  $p$, the Sobolev space $W_p^k(\Omega)$ is defined as
	\[
	W_p^k(\Omega):= \bigg\{ f: \|f\|_{k,L_p(\Omega)} := \bigg( \sum _{|\alpha|_1 \leq k} \|D^{\alpha} f\|_{L_p(\Omega)}^p \bigg) ^ {1/p} < \infty \bigg\}, \tag{12}
	\]
	for $p<\infty$, and as
	\[
	W_\infty^k(\Omega):= \bigg\{ f: \|f\|_{k,L_\infty(\Omega)} :=  \sum _{|\alpha|_1 \leq k}  \|D^{\alpha} f\|_{L_\infty(\Omega)}      < \infty \bigg\},\tag{12$^\prime$}
	\]
	for $p=\infty$.
\end{definition}
\begin{definition}\label{def:def12}
	Let $X = \{x_j\}_{j=1}^K$ be a set of points with  fill-distance $h$ and separation distance $\delta  = \underset{1 \leq i \neq j \leq N}{\minmath} \|x_i-x_j \|/2$. Then we say that $X$ is quasi-uniformly distributed if there exists a constant $\eta > 0$  independent of $X$ such that
	\begin{equation}
	\label{quasi_uniformly_dist} 2\delta \leq h \leq \eta \delta \,. \end{equation}
\end{definition}
\begin{definition}\label{def:def11}
	Let $\phi_w$ be a radial basis function, and let $\widehat \phi_w$ be its Fourier transforms. Define the supremum of the norm of $\widehat \phi_w$  as
	\begin{equation} \label{eq:def12}
	M_{\phi, w}(r):= \underset{\theta \in B(0,r) }\sup \| \widehat \phi_w(\theta) \| ^ {-1/2} \,.
	\end{equation}
\end{definition}
\begin{definition}\label{def:def10}
	Let $f$ be as defined in \eqref{RBF2}, and  $\phi_w$ be a radial basis function. Then the norm of the corresponding error functional is defined as 
	\begin{equation}
	P_{\phi, X}(x) = \underset{\|f\|_{\phi} \neq 0}{\sup} \frac{\|f(x) - \widetilde f(x)\|}{\|f\|_{\phi}}, \quad{\rm where}~
	\|f\|_{\phi} = \int_{\R^d} \frac{\|\widehat f (\theta)\|^2} {\widehat \phi_w(\theta)} d\theta \,.
	\end{equation}
\end{definition}

Finally, we are ready to state the promised theorem,   proven in \cite{yoon2001spectral}
\begin{Theorem}\label{thm:approx_order_RBF}
	Let $X = \{x_j\}_{j=1}^K$ be a set of quasi-uniformly distributed scattered points \textup{(}see Definition {\rm \ref{def:def12}}\textup{)}, and let $\widetilde f(x)$, defined as in \eqref{RBF}, be an interpolant to $f$ on $X$ using the radial basis function $\phi_w = \phi(\cdot/w)$. Let $M_{\phi, w}(r)$ with $r > 0$, be defined as in \eqref{eq:def12}. Assume that there exists a constant $\delta_0 > 0$ such that 
	\begin{equation}P_{\phi, X/w}(x/w) M_{\phi, w}(\delta_0/h) \leq o(h^k) \,. \end{equation} 
	Then, for every function $f \in W_{\infty}^k(\Omega)$, with $k\in \N$,   the error of the \textup{RBF} method is estimated as
	\begin{equation}\|f - \widetilde f\|_{L_{\infty}(\Omega)} = o(h^k)\,. \end{equation} 
\end{Theorem}

\begin{remark}
	As a result, an important key advantage  of the $\textup{RBF}$ method is that it performs better on quasi-uniform samples. This property will be utilized in the next section.
\end{remark}

We will also give a short introduction to the {Locally Weighted Average Approximation}, which will be used as a reference in the section devoted to our numerical examples. Given the values of the function $f: \mathbb{R}^n \to \mathbb{R}^s$ at the points $x_i \in \mathbb{R}^n$, the locally weighted average approximation of $f$ at a point $x$ is defined as 
\begin{eqnarray*} \label{locally_approx} f(x) = \frac{\sum_i w_i f(x_i)}{\sum_i w_i} \,,
\end{eqnarray*}
where $w_{i} =\exp\{- {\|x_i-x\|^2}/{h^2}\}$, and $h$ is the fill-distance of the points $\{x_i\}_{i=1}^I$. Concerning the accuracy of the method, we note that the locally weighted average approximation reconstructs constant functions.

%In this section we propose extending the RBF functions to deal with Function Approximation Over Manifold in High Dimension .
%We propose to utilize MLOP to create a noise-free quasi-uniform sampling of the manifold.

\section{Extending the MLOP Method to Approximation of Functions on a Manifold}

Let $\mathcal{M}$ be a smooth $d$-dimensional manifold in $\mathbb{R}^n$, where $d\ll n$. Let $P = \{p_j\}_{j=1}^J$ be a set of points which were sampled from $\mathcal{M}$ and are affected by noise, and let $f: \mathcal{M} \to \mathbb{R}^s$ be a smooth  function. Given noisy measurements of $f$ at the points in $P$, the approximation problem has two steps:
\begin{enumerate}[noitemsep]
	\item Find a noise-free representation of the manifold $\mathcal{M}$  and the noise-free values of $f$. 
	\item Estimate the value of the function $f$ at a new given point $x$. 
\end{enumerate}	

%Regression on manifolds, RBF

Our solution for the first step is based on generalizing the MLOP method designed for manifold denoising to the case of function denoising. The key idea of the solution is to embed the approximation problem in a higher-dimensional space, and denoise the data there. Given the input data, which consists of the set of points $\{p_j\}_{j=1}^J \subset \mathbb{R}^n$ and the set of values $\{f(p_j)\}_{j=1}^J \subset \mathbb{R}^s$, we define a new point-set $\widehat{P} = \{\widehat{p}_j\}_{j=1}^J$ to be the graph of a function $f$, i.e. as the set of ordered pairs, where $\widehat{p_j} = (p_j,f(p_j))$ is the pairing of $p_j$  with the value $f(p_j)$. The points in $\widehat{P}$ are now considered as data points in $\mathbb{R}^{n+s}$, taken from the  $\widehat{\mathcal{M}} = {\rm Graph}\,f=\{(x,f(x)):\, x\in \mathcal{M}\}$.   The newly defined set $\widehat{\mathcal{M}}$, being the graph of a smooth function $f$ defined on a smooth $d$-dimensional manifold,  is itself a smooth $d$-dimensional manifold. It should be noted that  prior to the embedding in the $(n+s)$-dimensional space, the values of $f$ should be normalized to the maximum value of the $p_j$ coordinates, in order to avoid them dominating the $p_j$ entries during the norm calculations of the MLOP algorithm. 

\forceindent In this setting, we are now denoising a $d$-dimensional manifold embedded in  $n+s$ dimensions. We apply the MLOP method on the new data set $\widehat{P}$, and look for a clean dataset $\widehat{Q}=\{\widehat{q}_i\}_{i=1}^I \subset \mathbb{R}^{n+s}$ which will serve as a noise-free approximation of $\widehat{\mathcal{M}}$. The main advantage of this approach is that with a single MLOP execution on the  $\{\widehat{p_j} = (p_j,f(p_j))\}$ data in $\mathbb{R}^{n+s}$ we produce a noise-free set $\widehat{Q} \subset \mathbb{R}^{n+s}$, which in fact consist of a noise-free set $Q$, that reconstructs the manifold $\mathcal{M}$, and an estimate of the clean value of the function evaluated at these points, $\widetilde f(Q)$.

\forceindent In the second step, we address the problem of evaluating the function at a new point $z \in \mathcal{M}$. The outcome of the first step   is a set of points which is not only noise-free, but also quasi-uniformly distributed on the manifold $\widehat{\mathcal{M}}$. This key idea paves the way towards estimating the value of the function at a new given point on $\mathcal{M}$, or near $\mathcal{M}$, with a good order of approximation. Our solution is based on utilizing the RBF approximation, as defined in \eqref{RBF}, while setting the centers at the cleaned points $Q$ and appending the corresponding cleaned function values $\widetilde f(Q)$. These steps are summarized in Algorithm \ref{alg:Alg3_}:

\begin{algorithm} [H]
	\caption{Function Approximation on a Manifold in High Dimensions}
	\vspace{1mm}
	\label{alg:Alg3_}
	\begin{algorithmic}[1]
		\State {\bfseries Input:} $P=\{p_j\}_{j=1}^J \subset \mathbb{R}^n$, $\{f(p_j)\}$, $\{z_k\}_{k=1}^K$ \hspace{35pt} $\triangleright$ where $\{z_k\}_{k=1}^K$ is a set of new points for function approximation
		\State{\bfseries Output:} $Q=\{q_i\}_{i=1}^I \subset \mathbb{R}^n$, $\widetilde f(Q)$, $\{\widetilde f(z_k)\}_{k=1}^K$
		\State Denoise the input data by running MLOP with $\widehat P = (P, f(P)) \subset \mathbb{R}^{n+s}$ $\to \widehat Q = (Q,  \widetilde f(Q))$
		\For {each $z_k \in \{z_k\}_{k=1}^K$ }
		\State Approximate $\widetilde f(z_k)$ via RBF, with centers set at $Q$ and the corresponding $\widetilde f(Q)$ values
		\EndFor
		
	\end{algorithmic}
\end{algorithm}
\vspace{-12mm}

\section{Theoretical Analysis of the Method}

In this section we discuss some of the theoretical aspects of the proposed approach to approximation of functions. Our analysis relies on the theory of the MLOP method as well as the RBF method. Specifically, we can use Theorems \ref{thm:conv}, \ref{thm:approx_order}, \ref{thm:convRate}, \ref{complex_MLOP} regarding the convergence of the MLOP algorithm to a stationary point, its rate of convergence, and its complexity, estimated for the problem at hand. We also build on the results about the order of approximation of the MLOP. Thus, we can state the following theorem on the order of approximation for the our approximation problem.

\subsection{Order of Approximation}

\begin{Theorem}[Order of approximation]\label{thm:approx_order_func_approx}
	Let $P=\{p_j\}_{j=1}^J$ be a set of points sampled from a $d$-dimensional $C^{2}$ manifold $\mathcal{M}$ without noise that satisfy the $h$-$\rho$ condition. Let $f: \mathcal{M} \to \mathbb{R}^s$ be a smooth multivariate function given at points  of $P$. Suppose that $f \in W_{\infty}^k(\Omega)$ and fulfills all the conditions of Theorem \textup{\ref{thm:approx_order_RBF}}. Then:
	\begin{enumerate}[label=\normalfont(\roman*)]
		\item For fixed $\rho$ and $\delta$, there is a set of points ${Q}=\{{q}_i\}_{i=1}^I \subset \mathbb{R}^{n}$ which approximates $\mathcal{M}$ with order $O(\widehat h^{2})$, where $\widehat h = \max\{\widehat h_1, \widehat h_2\}$ and $\widehat h_1$ and $\widehat h_2$ are as in Definition \textup{\ref{def:def2}} with respect to the high-dimensional data $\widehat P = \{(p_j, f(p_j))\}  \subset \mathbb{R}^{n+s}$. Moreover, the order of approximation of $f$ on the set $Q$ is also $O(\widehat h^{2})$.
		\item The order of approximation of $f$ at a new point is less than $C_1 \widehat h^2 +C_2 h_2^k$ when using the \textup{RBF} approximation with centers at $Q$, and $\widetilde f(Q)$, where $h_2$ is the optimal support of $\widehat w_{i,i'}$ as   in Definition \textup{\ref{def:def2}}, and $C_1$ and $C_2$ are constants.
	\end{enumerate}	
	
\end{Theorem}

\begin{proof}

	Given the input data, which consists of the points $\{p_j\}_{j=1}^J$ and the function values $\{f(p_j)\}_{j=1}^J$, we define a new point-set, which is the graph of the function $f$, $\widehat{P} = \{\widehat{p}_j\}_{j=1}^J$ sampled from a new manifold $\widehat{\mathcal{M}}  \in \mathbb{R}^{n+s}$, where $\widehat{p_j}$ are defined as the pairing of  the $p_j$-data with the corresponding values $f(p_j)$. In this setting, we use the MLOP algorithm to denoise a $d$-dimensional manifold embedded in $(n+s)$-dimension. Thus, by applying the MLOP method to the new data we obtain the set $\widehat Q$ of points that reconstruct $\widehat{\mathcal{M}}$. From Theorem \ref{thm:approx_order} it follows that $\widehat Q$ approximates $\widehat{\mathcal{M}}$ with the order $O(\widehat h^{2})$, where $\widehat h$ is the representative distance introduced in Definition  \ref{def:def2} for the extended data set $\widehat P = \{(p_j, f(p_j))\}  \in \mathbb{R}^{n+s}$. 
	Recall that $\widehat{Q}$ is in fact a combination of a noise-free set $Q$, which reconstructs the manifold $\mathcal{M}$, and an estimate of the clean values of the function at these points in $Q$, $\widetilde f(Q)$. Therefore, the order of approximation of $Q$ to $\mathcal{M}$, and of the estimated values $\widehat f$ to $f$ is $O(\widehat h^{2})$.

	\forceindent Subsequently, we use an RBF approximation with theoretical order of approximation $O(h^k)$ (as stated in Theorem \ref{thm:approx_order_RBF}). In our case the relevant $h$ is $h_2$, the support size  of $\widehat w_{i,i'}$ as  in Definition \ref{def:def2}. Therefore, the overall order of approximation for a new point is a combination of the two orders, namely $\le C_1 \widehat h^2 +C_2h_2^k$, with $C_1$, and $C_2$ constants.
\end{proof}

\subsection{Complexity of the Approximation of Functions}
\begin{Theorem}
	\label{complex_MLOP_func_approx}
	Let $P=\{{p_j }\}_{j =1}^J$ be a set of points  sampled near a $d$-dimensional manifold $\mathcal{M} \subset \mathbb{R}^n$ and let $Q=\{q_i \}_{i=1}^I$ be a set of points which will provide the desired noise-free manifold reconstruction. Let $f: \mathcal{M} \to \mathbb{R}^s$ be a multivariate function given at the points of $P$. Then the complexity of the approximation of $f$ via the \textup{MLOP} algorithm for the denoising step is $O((n+s)mJ + k I((n+s)m\widehat{I}+\widehat{J}) + nmI + I\, \textup{log}_2^2 s)$, and for evaluating $f$ at a new point is $O(nm+I)$, where the number of iterations $k$ is bounded as in Theorem \textup{\ref{thm:convRate}}, $m$, with $m\ll n$, is the smaller dimension to which we reduce the dimension of the data, and $\widehat{I}$ and $\widehat{J}$ are the numbers of $Q$-points and $P$-points, respectively,  in the support of the weight function $\hat w_{i,i'}$ and $w_{i,j}$.
\end{Theorem}

\begin{proof} 
	
	The estimate of the complexity of the algorithm can be separated into two steps: \textbf{pre-processing} and \textbf{evaluating the function at a new point}. The pre-processing step consists of applying the MLOP algorithm in $\R^{n+s}$, as well as finding the RBF coefficients $\lambda_i$ by solving the Least-Squares problem. By Theorem \ref{complex_MLOP}, the complexity of applying the MLOP in the higher dimension $n+s$ is $O((n+s)mJ + k I((n+s)m\widehat{I}+\widehat{J}))$. The output of this stage is a points set $Q$ of size $I$, and the corresponding set of values $\widetilde f(Q)$. From this point on, all the approximation operations are performed on $Q$, and if $I\ll J$ then we can increase sufficiently the efficiency. The next part of the pre-processing step is to evaluate the radial basis function $\phi$ for each $q_i \in Q$, which costs $O(nmI)$, and then to find the $\lambda_i$ by solving the 
	Least-Squares problem, which takes $O(I\, \textup{log}_2^2 s)$ (as shown in \cite{li1996new}). As a result, the complexity of the pre-processing step is $O((n+s)mJ + k I((n+s)m\widehat{I}+\widehat{J}) + nmI + I \,\textup{log}_2^2 s)$. 
	
	\forceindent Finally, using the $\lambda_i$ already found, we evaluate the function at a new point in time $O(nm+I)$. It should be stressed that although the pre-processing steps are cost-effective, they are executed once before the function is approximated at a new points set. Thus, if the number of new points for which the approximation needs to be found is large, then the pre-processing steps have less effect on the runtime.
\end{proof}
%In what follows we illustrate the effect of 

%Least squares complexity summary
%---------
%n, d = J
%J>> n
%100*1000*7 + 1000^3
%100^2 * 1000
%
%LS
%O(n^2m) - the default 
%2mn^2, n- dim of X_i, m - how many equations there are 
% O(nd log(d) + d^3) log(1/\epsilon) via sketching \cite{pilanci2016iterative}, where d=m
% O(n log_2^2 m). \cite{li1996new}

%following {Local error estimates for radial basis function interpolation of scattered data, ZONG-MIN WU}
%$\phi(r) = r^s$, the approximation error is O(h^{s/2})
%in our case s = 0.4
%thus we have O(h^0.2)

\section{Numerical Examples}

In what follows we present several numerical experiments to demonstrate the advantages of our methodology. We can point out two strengths of the proposed approximation approach. On the one hand, denoising the data domain as well as the function codomain plays an important role in the approximation of functions. On the other hand, sampling the manifold quasi-uniformly  improves significantly the approximation of functions by means of classical approximation methods on new data.

\forceindent Given data sampled from a manifold with noise, and the noisy values of a function $f$ at these points, we follow the function approximation procedure described above. Specifically, we define a new problem in $\R^{n+s}$, and apply $k$ MLOP iterations  to clean the newly defined manifold, which results in a new point-set $Q^{(k)}$, and the corresponding cleaned value set $\widetilde f(Q^{(k)})$. Next, we randomly select 100 points, $\{z_{i}\}_{i=1}^{100}$, from a clean reference dataset, and estimate the values of the  function $f$ at these points using both the RBF approximation, where the centers of the RBF function are taken at the  points of $Q^{(k)}$  (with the radial basis function set to either $\phi_1$, $\phi_2$, or $\phi_3$, as defined below  equation \eqref{RBF}), and the locally weighted average approximation (defined by formula \ref{locally_approx}). As a result, the approximation at the new points relies on the clean quasi-uniformly distributed $Q^{(k)}$ points, as well as on the clean values $\widetilde f(Q^{(k)})$. In all the stages above we evaluate the accuracy of the approximation as the relative maximum error of the $L_1$ norm of the difference between the value of $\widetilde f$ at the new point and the value of $f$ at the closest point in  the reference dataset, as well as the root-mean-square  error and the standard deviation.

\forceindent We start with two examples of functions, one smooth and the other non-smooth, both on a one-dimensional manifold embedded in  a high-dimensional space. Although in principle the approximation requires a smooth function, it can still be applied to a non-smooth function, provided that we end up with a smoothed result. Specifically, we consider the case of the manifold $O(2)$ of orthogonal matrices, embedded in  a 60-dimensional linear space by using the parameterization 
\begin{equation}
\label{eq:O2_eq} p=[\cos (\theta), -\sin(\theta), \sin(\theta), \cos(\theta), 0, \dotsc, 0]\,,\end{equation} 
where $\theta \in [-\pi, \pi]$. 
The input dataset $\widehat P$ was constructed by sampling 500 equally distributed points in the parameter space. Next, we randomly sampled an orthogonal matrix $A \in \R^{60 \times 60}$, and created a new point-set via the non-trivial vector embedding
\begin{equation} 
P = A \widehat P \,.
\end{equation}
Subsequently, we added a uniform noise $U(-0.1, 0.1)$, and initialized the set $Q$ by selecting $55$ points from $P$. Figure \ref{fig:func_approx_sin} (A) illustrates the first two coordinates of the points in our set (after multiplication by the matrix $A^{-1}$). The noisy sample points are shown in green, while the initial reconstruction points are shown in red. 

\forceindent We start by approximating the \textbf{smooth function} $f(x) = \frac{1}{4} (1+\sin(10\,\theta))$, where $\theta$ corresponds to the value used in the expression \eqref{eq:O2_eq} of $p$. Next, a uniform noise $U(-0.1, 0.1)$ was added in the codomain (see Figure \ref{fig:func_approx_sin} (C)), and then we applied the MLOP algorithm, which reconstructed the manifold (Figure \ref{fig:func_approx_sin} (B)), as well as  denoised function values (Figure \ref{fig:func_approx_sin} (D)). Table \ref{func_approx_errors_table} summarizes the errors that correspond to different scenarios. We first notice that due to the denoising effect on $Q$ points the maximum relative error decreased from $0.31$ to $0.13$ for noisy data $Q^{(0)}$ as opposed to the clean data $Q^{(150)}$. Next, we can also see the benefits of using the MLOP algorithm prior to  approximating  the function with the RBF method on the new data. In the present example the maximum relative error of the best RBF execution decreased dramatically from $0.66$ when running the RBF with centers at $Q^{(0)}$, to an $0.12$ when running on the quasi-uniform data, and with very low error variance. A quick comparison between the RBF method and the approximation via locally weighted average shows that the latter loses the battle to RBF (even though it produce better results on the clean data versus the noisy one).

% points in the support #P = 44, #Q = 15
% h1 = 0.18, h2 = 0.6720

\begin{figure}[H]
	\centering
	\captionsetup[subfloat]{farskip=0pt,captionskip=0pt, aboveskip=0pt}
	\subfloat[][]{ \includegraphics[width=0.5\textwidth]{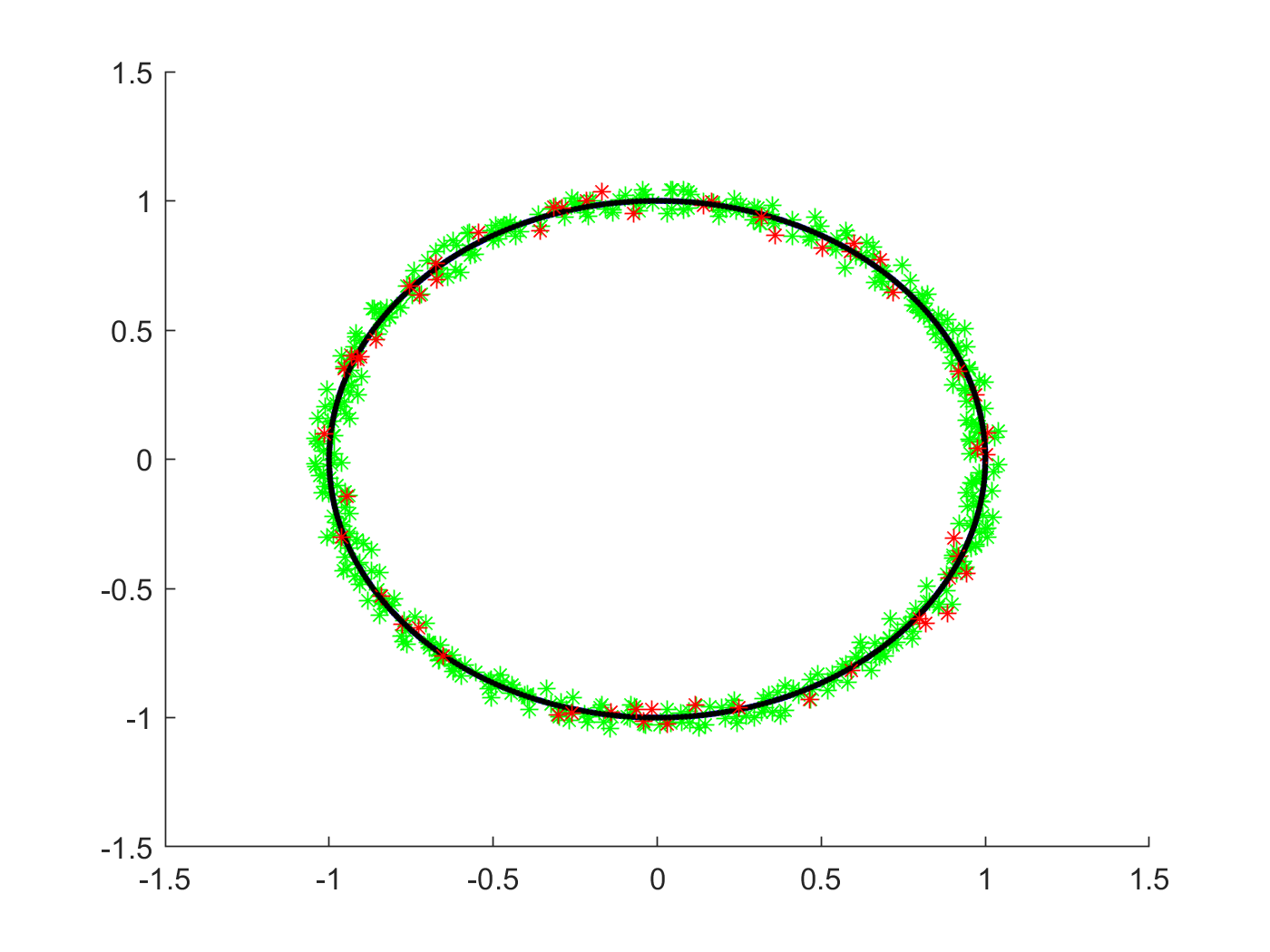} } \hspace{-2em}
	\subfloat[][]{ \includegraphics[width=0.5\textwidth]{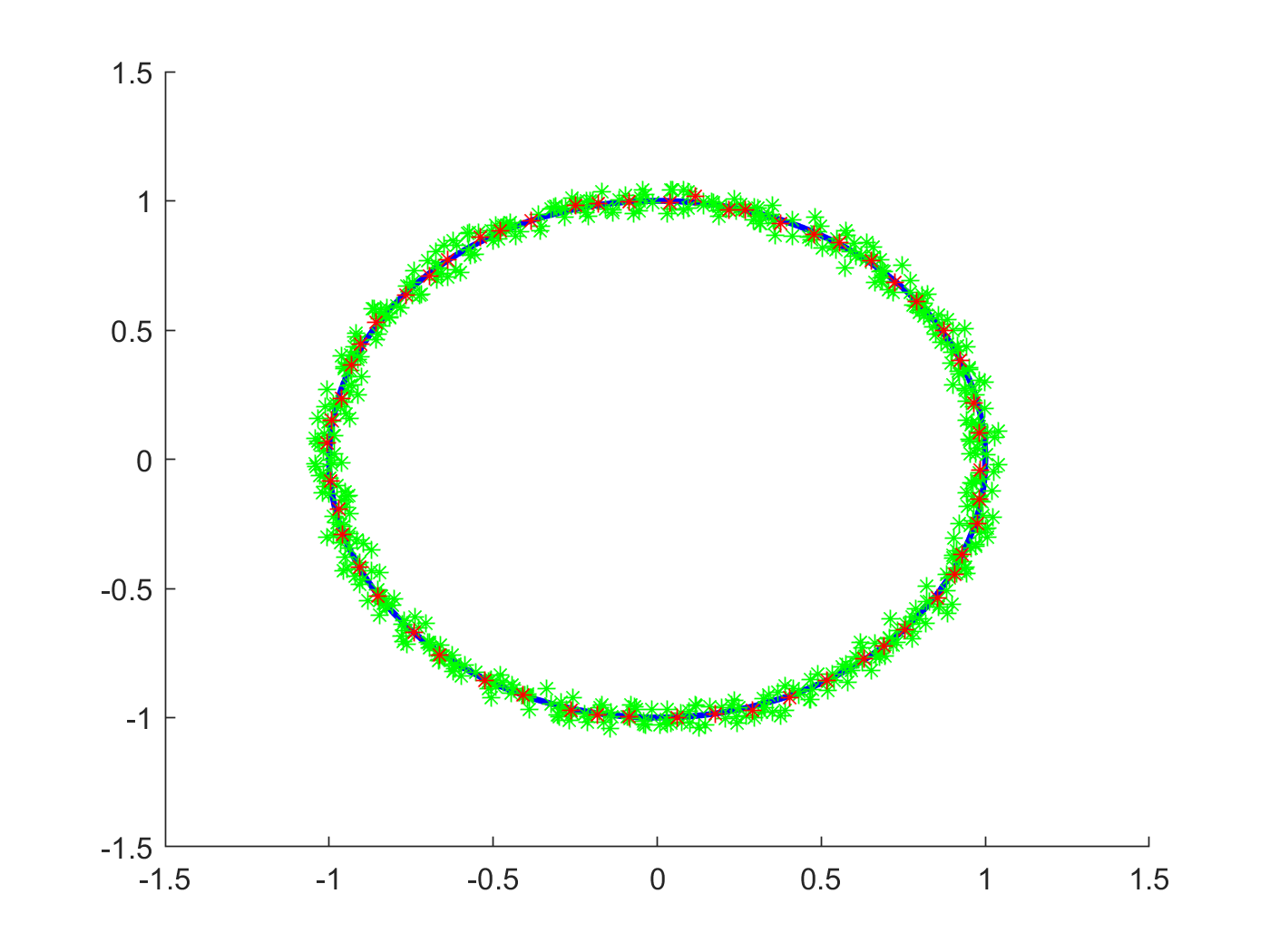} } \\[-0.7ex]
	\subfloat[][]{ \includegraphics[width=0.5\textwidth]{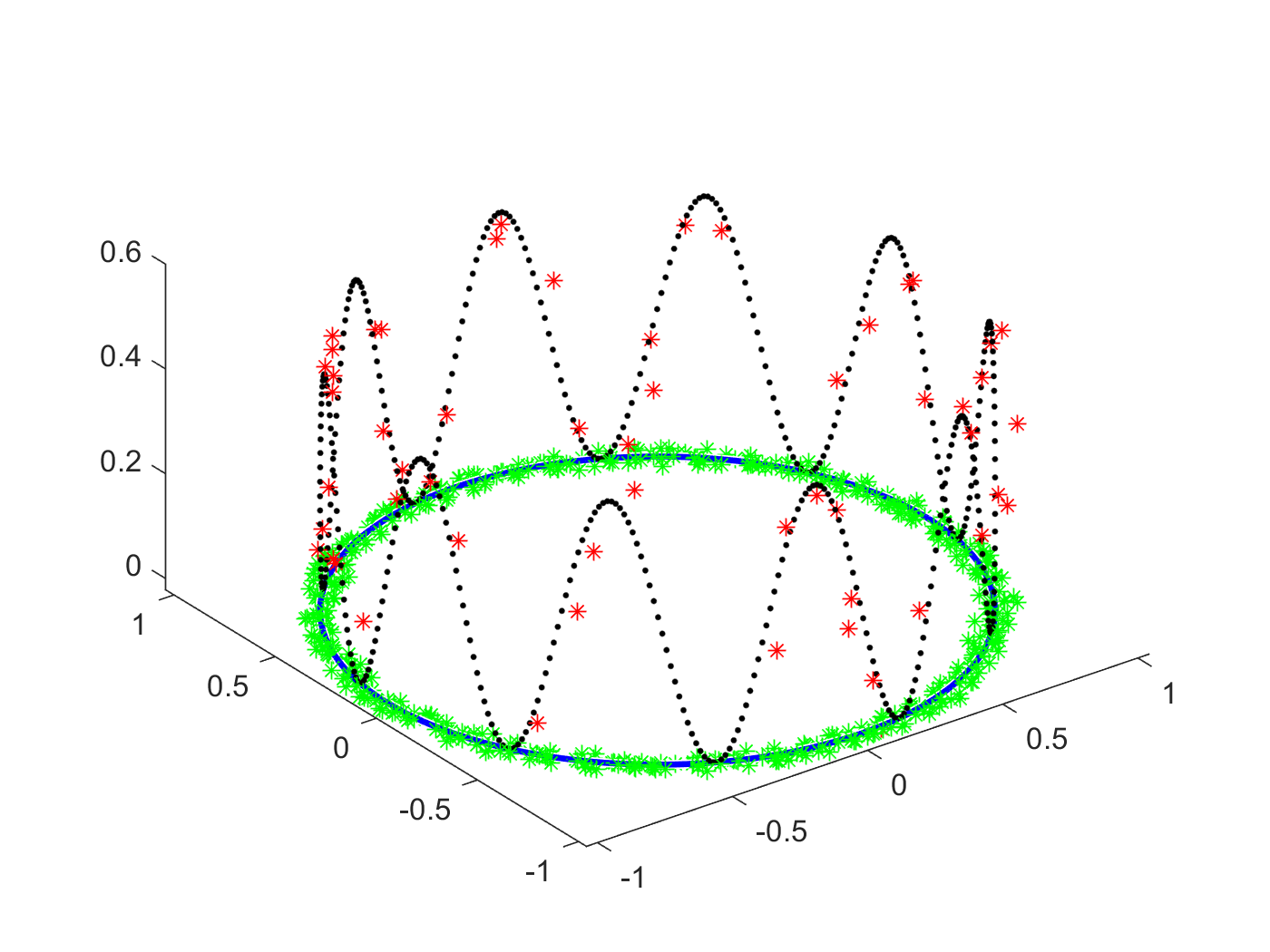} } \hspace{-2em}
	\subfloat[][]{\includegraphics[width=0.5\textwidth]{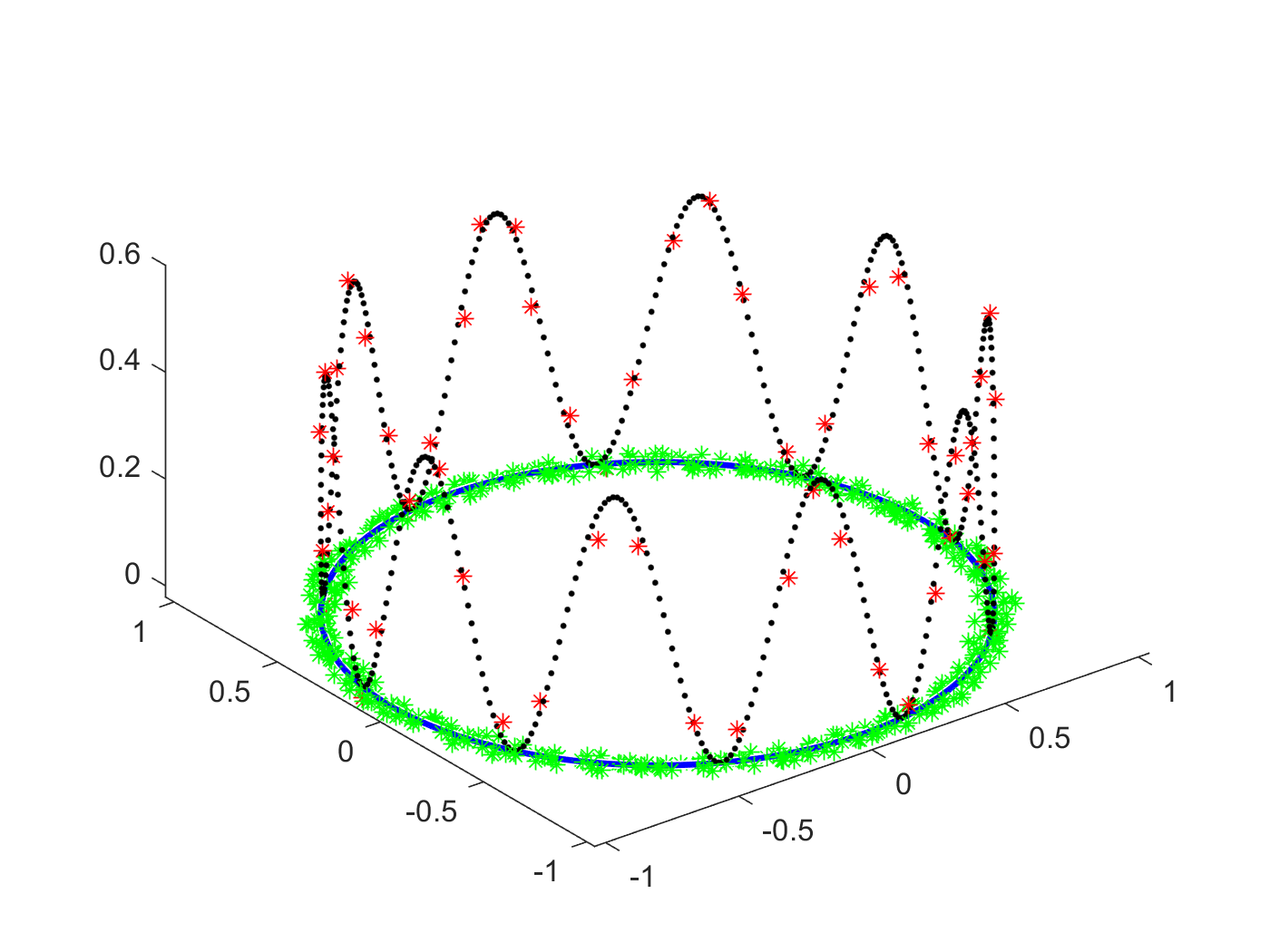}}
	\caption{Manifold of orthogonal matrices embedded in a 60-dimensional space. Shown are the first two coordinates of the point-set (after multiplication by $A^{-1}$). (A) Scattered data with uniformly distributed noise $U(-0.1; 0.1)$ (green), and the initial point-set $Q^{(0)}$ (red). (B) The resulting point-set of the MLOP algorithm after 150 iterations, $Q^{(150)}$ (red), overlaying the noisy samples (green). (C) The initial function values evaluated at the original point-set $Q^{(0)}$ with noise $U(-0.1, 0.1)$. The black line shows the noise-free reference data. (D) Smooth function approximation via the MLOP algorithm  at the data points $Q^{(150)}$.}
	\label{fig:func_approx_sin}
\end{figure}

\forceindent We then applied the approximation procedure to the \textbf{non-smooth} function given by $f(x) = \frac{1}{6}(1+\arccos(\cos(10\,\theta)))$. We evaluated the function at the $P$-points and added the uniform noise $U(-0.1, 0.1)$ (see Figure \ref{fig:func_approx_zigzag} (A)). Then, we applied the MLOP algorithm, which resulted in a reconstructed manifold, as well as denoised function values; see Figure \ref{fig:func_approx_zigzag} (B). As this figure shows,  %Figure \ref{fig:func_approx_zigzag} (B), 
the non-smooth function $f$ is approximated reliably. This is also reflected in the errors listed in Table \ref{func_approx_errors_table}, which shows that the approximation error decreased from $0.2$ to $0.1$  after the denoising procedure. The advantages of the MLOP approach are also demonstrated by the error decrease, for the best choice of radial basis function, from $0.77$ to $0.15$ for approximation on a new point-set. This shows the robustness of the approximation process with respect to the clean data. Here again, we see that the RBF method produces a better approximation then the weighted average.
\vspace{-12mm}
\begin{figure}[H]
	\centering
	\captionsetup[subfloat]{farskip=0pt,captionskip=0pt, aboveskip=0pt}
	\subfloat[][]{ \includegraphics[width=0.5\textwidth]{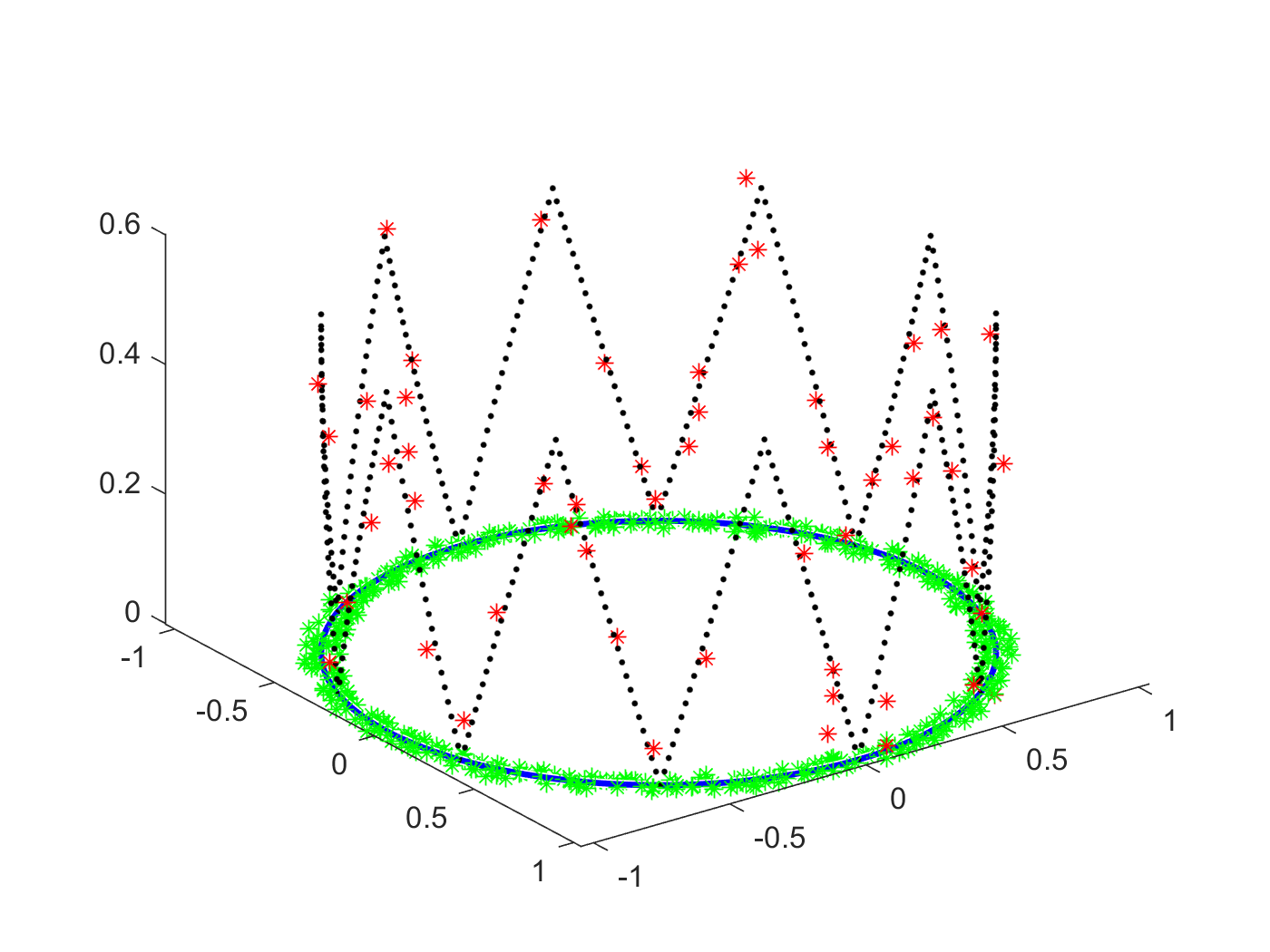} } \hspace{-4em}
	\subfloat[][]{\includegraphics[width=0.5\textwidth]{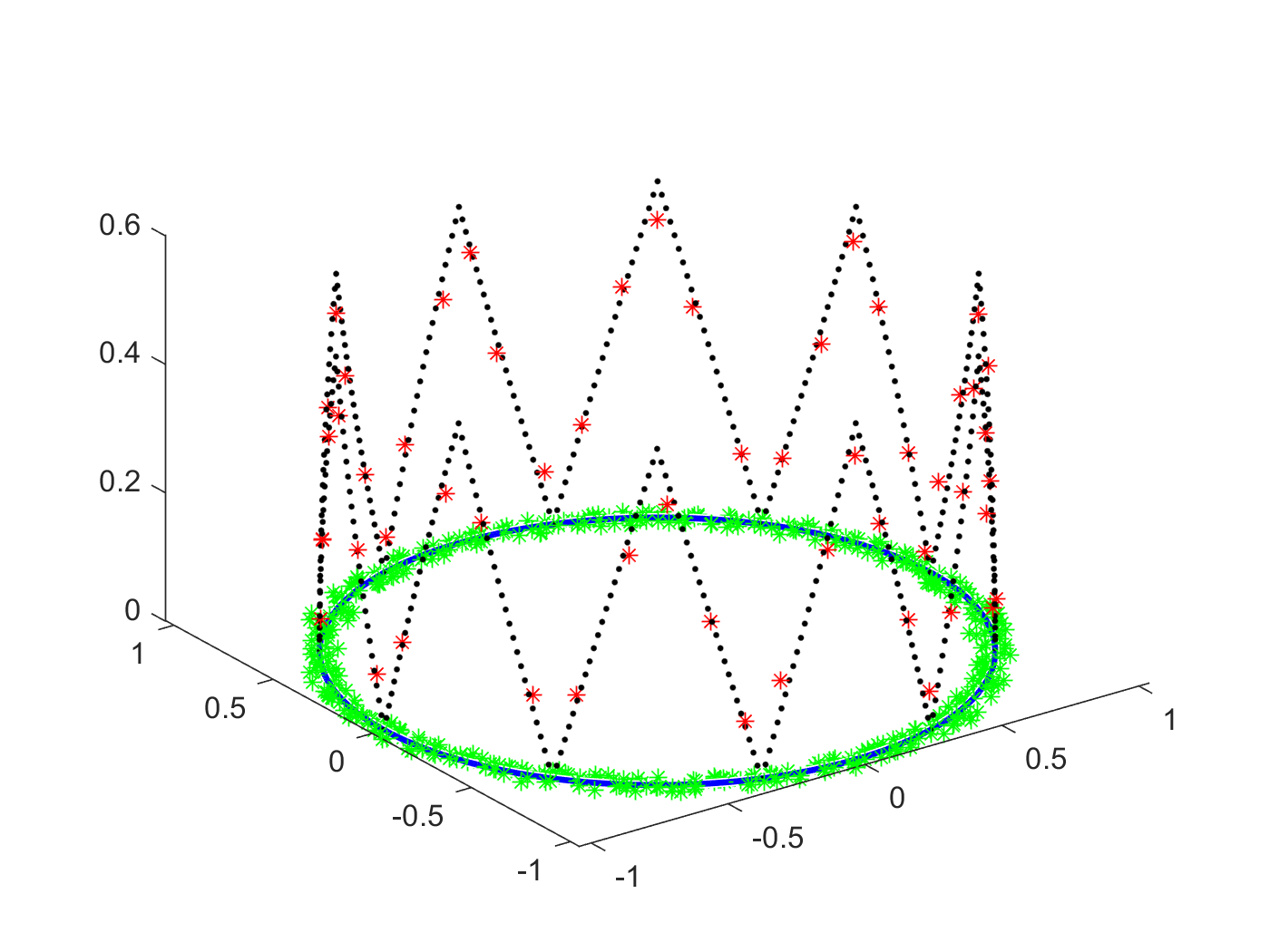}}
	\caption{Manifold of orthogonal matrices embedded in  a 60-dimensional space. Shown are the first two coordinates of the point-set (after  multiplication by $A^{-1}$). Scattered data with uniformly distributed noise $U(-0.1; 0.1)$ (green), and the $Q$ point-set (red). Left: The initial function values evaluated at the original $Q^{(0)}$-points with noise $U(-0.1, 0.1)$. The black line shows the noise-free reference data. Right: Approximation of our non-smooth function via MLOP at the data points $Q^{(150)}$.}
	\label{fig:func_approx_zigzag}
\end{figure}

\begin{table}[H]
	\linespread{1.1}\selectfont\centering
	\caption{Summary of the maximum and root mean squared errors with standard deviations errors of approximation of functions on the $O(2)$ manifold embedded into 60-dimensional space}
	\begin{tabular}{p{5cm}p{2.3cm}p{2.5cm}p{2.3cm}p{2.5cm}} %{llclc}
		\hline
		& \multicolumn{2}{c|}{$f(x) = \frac{1}{4}(1+\sin 10 x)$}               & \multicolumn{2}{c}{$f(x) = \frac{1}{6}(1+\arccos(\cos 10 x))$}      \\ \cline{2-5} 
		& Max relative error & \multicolumn{1}{c|}{RMSE $\pm$ var}  & Max relative error & RMSE $\pm$ var                      \\ \hline
		\textbf{Error over $Q^{(k)}$}                 & \textbf{}                     & \multicolumn{1}{l}{\textbf{}}        & \textbf{}                     & \multicolumn{1}{l}{\textbf{}}       \\ \hline
		$f(Q^{(0)})$                                                   & 0.31                          & $0.06 \pm 0.0011$ & 0.2                          & $0.04 \pm 0.0007$                     \\
		$f(Q^{(150)}) $                                              & 0.13                         & $0.03 \pm 0.0003$ & 0.1                         & $0.03 \pm 0.0002$                     \\ \hline
		\multicolumn{5}{l}{\textbf{Error over 100 new points}}                                                                                                                                   \\ \hline
		RBF, $\phi_1$, centers at $Q^{(0)}$, noisy $f$             & 0.66                          & $0.14 \pm 0.0079$ & 0.77                          & $0.16 \pm 0.010$                      \\
		RBF, $\phi_1$, centers at $Q^{(150)}$, cleaned $\widetilde f$ & 0.12                          & $0.03 \pm 0.0002$ & 0.15                         & $0.03 \pm 0.0004$                     \\
		RBF, $\phi_2$, centers at $Q^{(150)}$, cleaned $\widetilde f$ & 0.12                          & $0.03 \pm 0.0002$                      & 0.16                         & $0.03 \pm 0.0003$ \\
		RBF, $\phi_3$, centers at $Q^{(150)}$, cleaned $\widetilde f$ & 0.24                          & $0.05 \pm 0.0009$                      & 0.21                          & $0.05 \pm 0.0007$ \\
		Weighted average                                           & 0.28                          & $0.10 \pm 0.0018$ & 0.32                          & $0.09 \pm 0.0024$                     \\ \hline
	\end{tabular}
	\label{func_approx_errors_table}
\end{table}
\vspace{-12mm}

\forceindent In what follows we demonstrate our function approximation methodology on several examples of a low-dimensional manifold embedded in high-dimensional space. Specifically, we embedded a two-dimensional cylindrical structure and then a six-dimensional cylindrical structure in  $\R^{60}$. We start with the two-dimensional cylindrical structure. We sampled the structure using the parameterization
\begin{eqnarray*}
	p=t v_1+\frac{R}{\sqrt 2}\big(\cos(u)v_2+  \sin(u)v_3\big)\,,
\end{eqnarray*}
where $v_1=[1,1,1,1,1, \dotsc ,1]$, $v_2=[0,1,-1,0,0, \dotsc ,0], v_3=[1,0,0,-1,0, \dotsc ,0]$ 
$(v_1,v_2,v_3 \in \R^{60})$, $t \in [0,2]$ and $u \in [0.1 \pi,1.5 \pi]$. Using this representation, $800$ equally distributed (in parameter space) points were sampled with uniformly distributed noise (i.e., $U(-0.1, 0.1)$). We evaluated the function $f(t, u) = 1.3(1+\textup{sin}(0.5 u+1.5 t))$ at these points, and constructed the initial $Q$-set by randomly sampling $150$ points (see Figure \ref{fig:func_approx_cylinder} (A) for the $P$ and $Q$ data, and Figure \ref{fig:func_approx_cylinder} (C) for the values of the function at the 
$Q$-points). The representative distances of the $P$-set and $Q$-set were $h_1 = 0.19$ and $h_2 =0.27$, resepctively.  Next, we applied the MLOP algorithm on the new data of $(P, f(P))$, and extracted the sets $Q$  and $\widetilde f(Q)$ (see Figure \ref{fig:func_approx_cylinder} (B) for the cleaned $Q$-points, and Figure \ref{fig:func_approx_cylinder} (D) for the cleaned values at the 
$Q$-points). In addition, we approximated the values of $f$ at   100 new points, randomly selected from the reference data. The evaluation error results are summarized in Table \ref{3D_6Dcylinder}. The maximum relative $L_1$ error and the RMSE accompanied with the variance are summarized in Table \ref{3D_6Dcylinder}. It should be noted  that since we compare the maximum error to clean data, the relative error can exceed 1. One can see that the new data RBF with $\phi_2$ as well as $\phi_3$ achieve the lowest errors.

\vspace{-10mm}
\begin{figure}[H]
	\centering
	\captionsetup[subfloat]{farskip=0pt,captionskip=0pt, aboveskip=0pt}
	\subfloat[][]{ \includegraphics[width=0.5\textwidth]{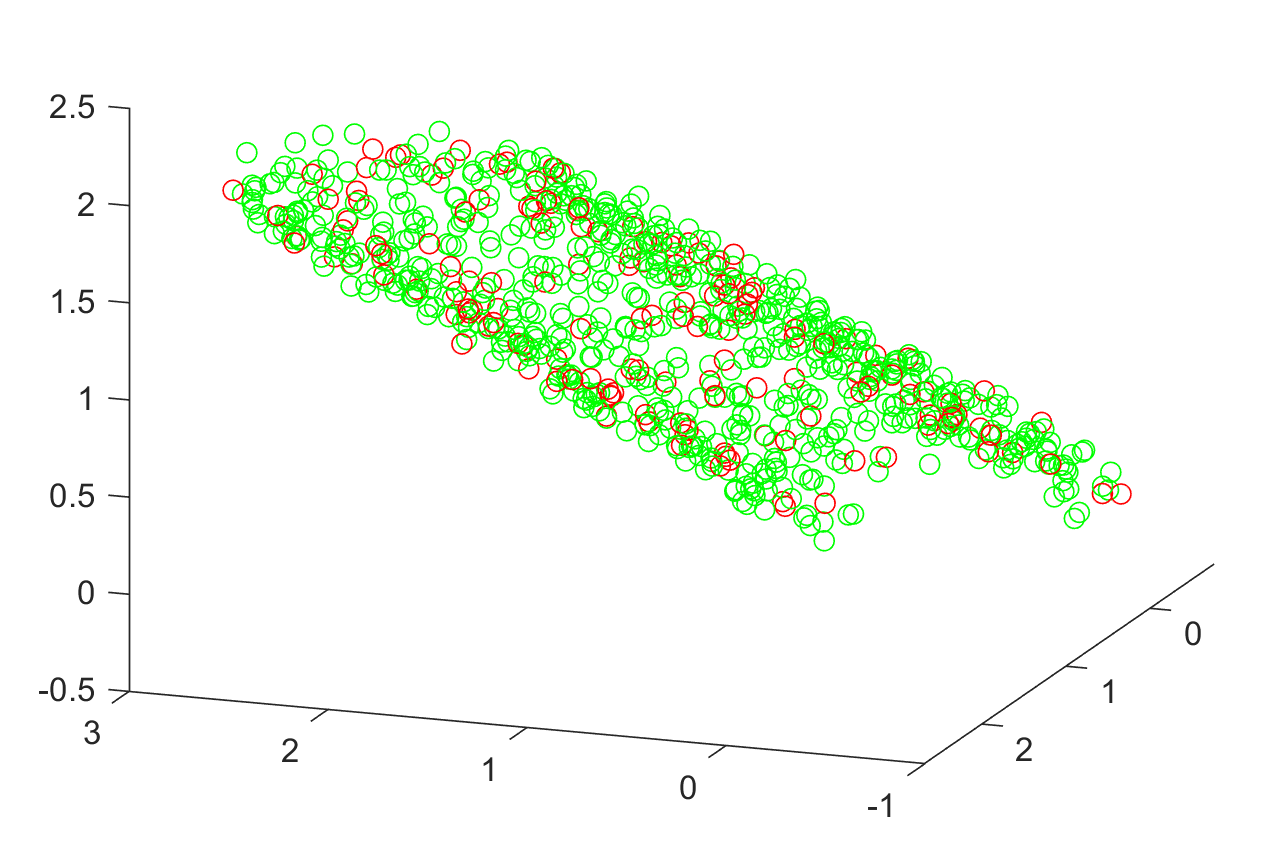} } \hspace{-2em}
	\subfloat[][]{ \includegraphics[width=0.5\textwidth]{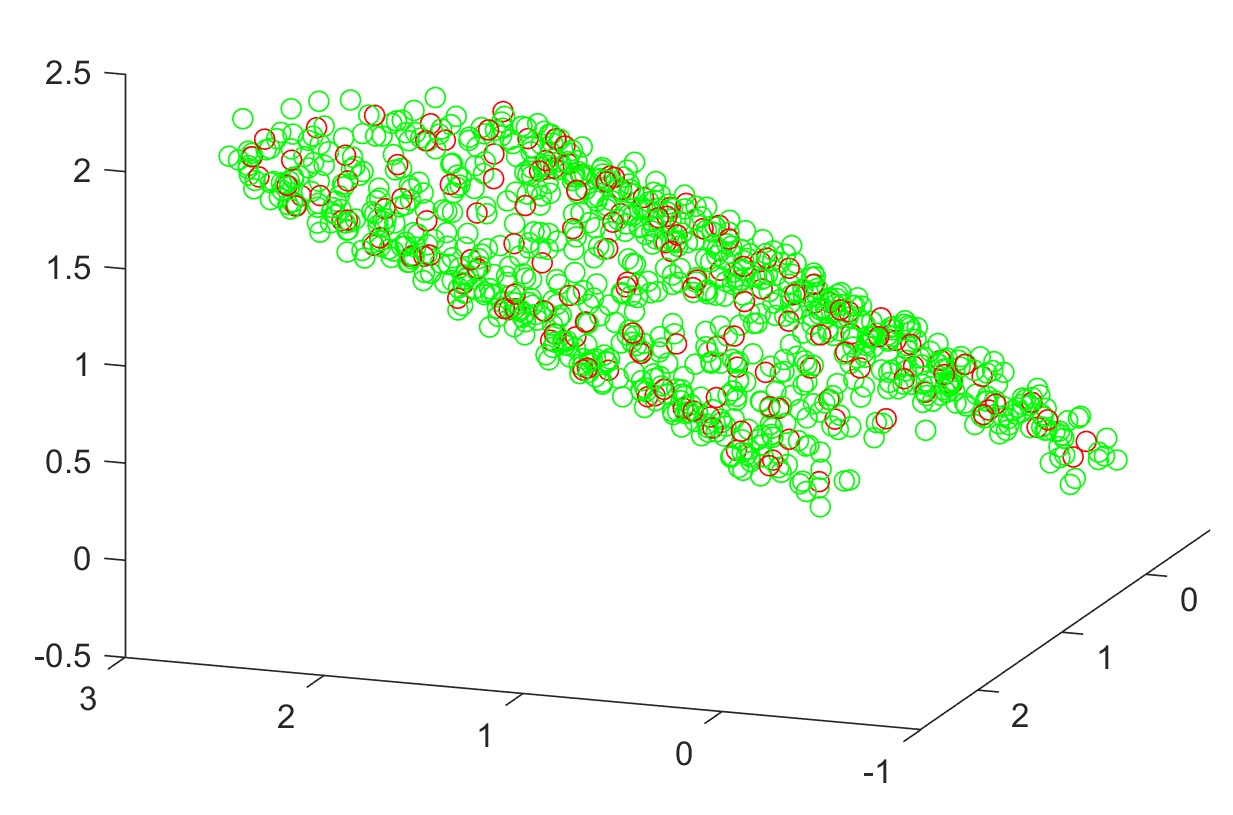} } \\[-0.7ex]
	\subfloat[][]{ \includegraphics[width=0.5\textwidth]{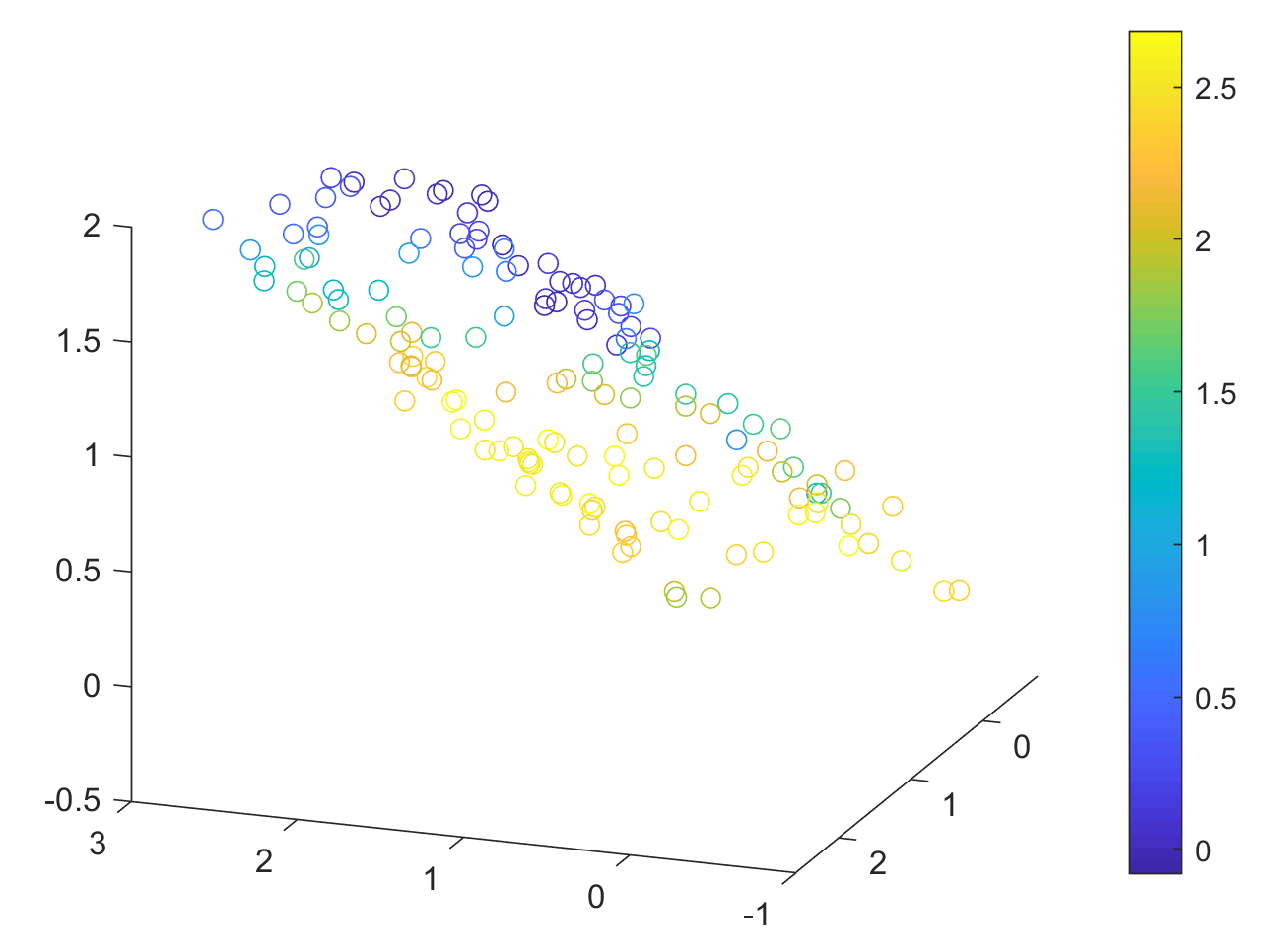} } \hspace{-2em}
	\subfloat[][]{\includegraphics[width=0.5\textwidth]{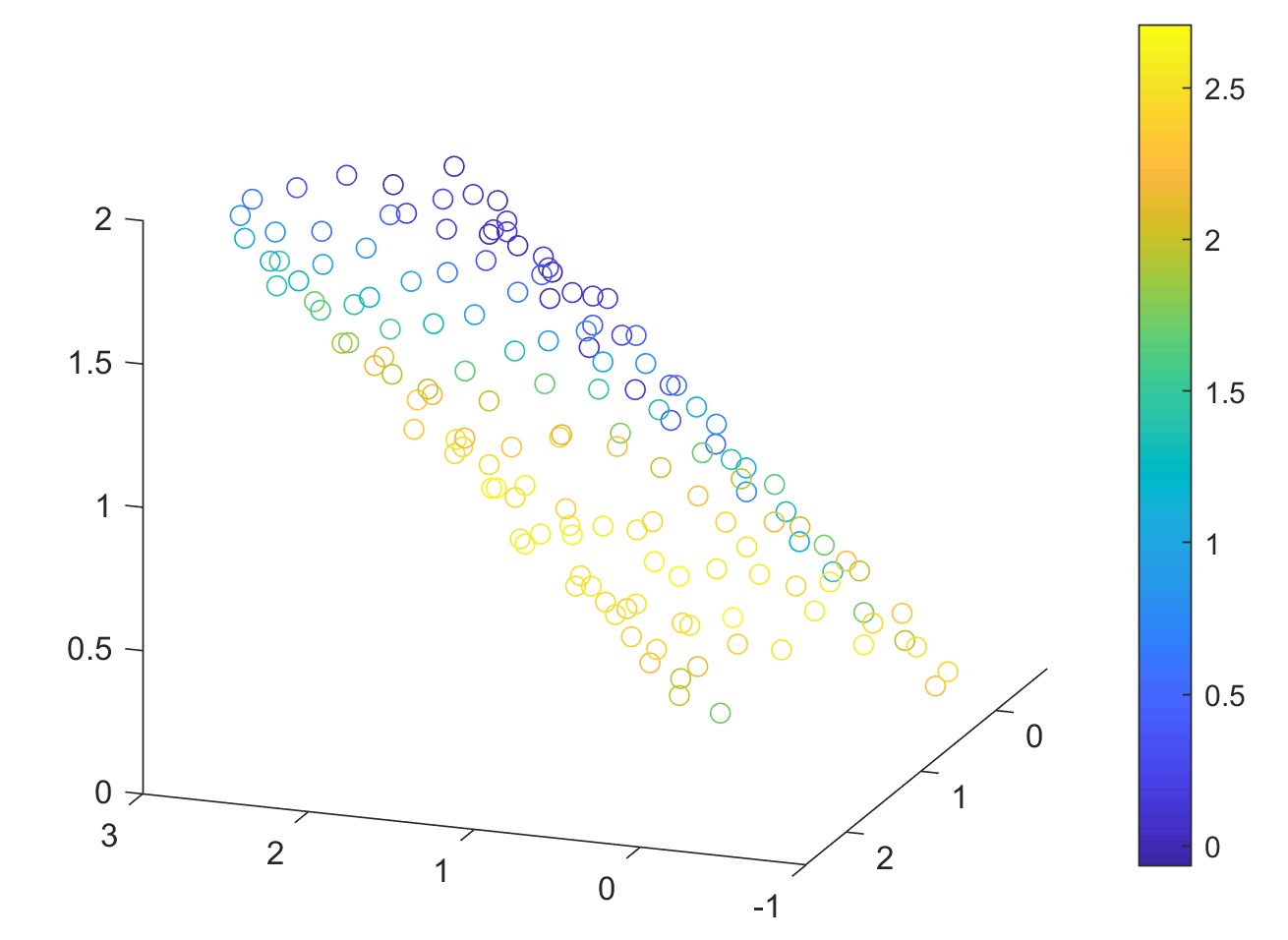}}
	\caption{Two-dimensional cylindrical structure embedded in  a 60-dimensional space. The first three coordinates of the point-set are shown. (A) Scattered data with uniformly distributed noise $U(-0.1; 0.1)$ (green), and the initial point-set $Q^{(0)}$ (red). (B) The resulting point-set of the MLOP algorithm after 200 iterations, $Q^{(300)}$ (red), overlaying the noisy samples (green). (C) The initial  values of the function at the original $Q^{(0)}$-points with noise $U(-0.1, 0.1)$. (D) MLOP approximation at the data points $Q^{(300)}$.}
	\label{fig:func_approx_cylinder}
\end{figure}

\subsection*{Six-dimensional cylindrical structure}

\label{sec:10dimCylinder_example_1}
Next, we tested our method on higher-dimensional manifolds  by utilizing an $n$-sphere to generate an $(n+1)$-dimensional cylinder (in the example of the two-dimensional cylinder, we used a circle  to generate the structure). Here, we utilized a five-dimensional sphere to build a six-dimensional manifold, using the parameterization
\begin{align*}
x_1 = R \cos(u_1) \,,\quad
x_2 = R \sin(u_1)\cos(u_2),\quad 
\ldots,\quad 
x_{6} = R \sin(u_1)\sin(u_2)\cdots\sin(u_5)\sin(u_6)\,. 
\end{align*}
We then embedded the sampled data in a 60-dimensional space by the parametrization
\begin{eqnarray}
p=t v_0 + R^2 [x_1,x_2, x_3, x_4, x_5, x_6, 0, \dotsc ,0]\,,
\end{eqnarray}
where $R = 1.5$, $t \in [0,2]$, $u_i \in [0.1 \pi, 0.6 \pi]$, and $v_0 \in \R^{60}$ is a vector with 1's in positions $1,...,d+1$ and 0 in the remaining positions. We randomly sampled the $P$-points from the six-dimensional cylindrical structure and embedded the sampled data in  a 60-dimensional space. This process resulted in 1200 points in the $P$-set. We evaluated the function $f(u_1, \ldots, u_6) = \sum\limits_{i=1}^6 {u_i}$ at these points, and constructed the initial $Q$-set by randomly sampling $460$ points. Next, uniformly distributed noise $U(-0.2; 0.2)$ was added to the points. To avoid trying to visualize a six-dimensional manifold, we plot here the cross-section of the cylindrical structure in three dimensions. In Figure \ref{fig:func_approx_7Dcylinder} (A) we present the $P$- and $Q$-data, and in Figure  \ref{fig:func_approx_7Dcylinder} (C) the  values of the function at the $Q$-points. The representative distances of the $P$-set and $Q$-set were $h_1 = 0.24$ and $h_2 =0.37$, respectively. We  then applied the MLOP algorithm on the new data of $(P, f(P))$, and extracted the sets $Q$ and $\widetilde f(Q)$ (see Figure \ref{fig:func_approx_7Dcylinder} (B) for the cleaned $Q$-points, and Figure \ref{fig:func_approx_7Dcylinder} (D) for the cleaned function values at the $Q$ points). In addition, we approximated the values of the function at 100 new  points, randomly selected from the reference data. The maximum relative $L_1$ error and the RMSE accompanied with the variance are summarized in Table \ref{3D_6Dcylinder}. One can see that the new data RBF with $\phi_2$ and the weighted average achieve lower errors.

\vspace{-10mm}

\begin{figure}[H]
	\centering
	\captionsetup[subfloat]{farskip=0pt,captionskip=0pt, aboveskip=0pt}
	\subfloat[][]{ \includegraphics[width=0.5\textwidth]{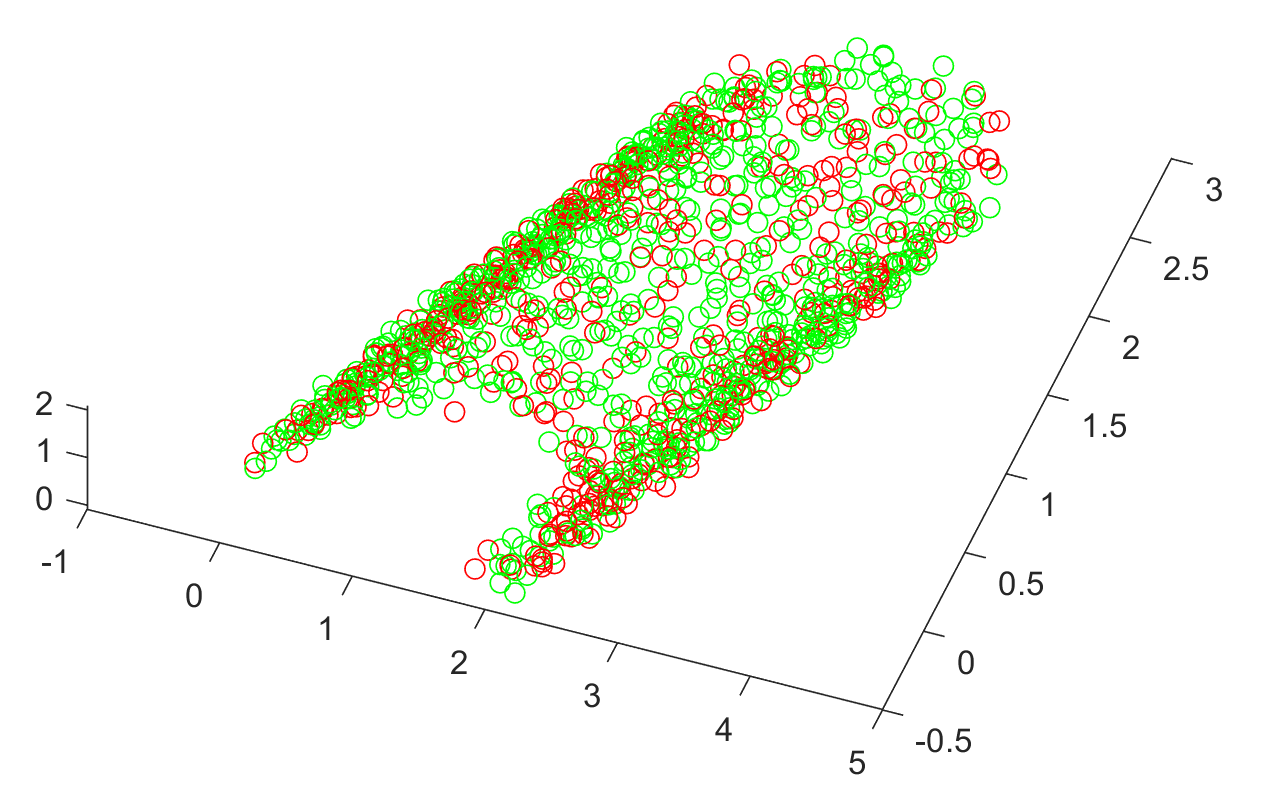} } \hspace{-2em}
	\subfloat[][]{ \includegraphics[width=0.5\textwidth]{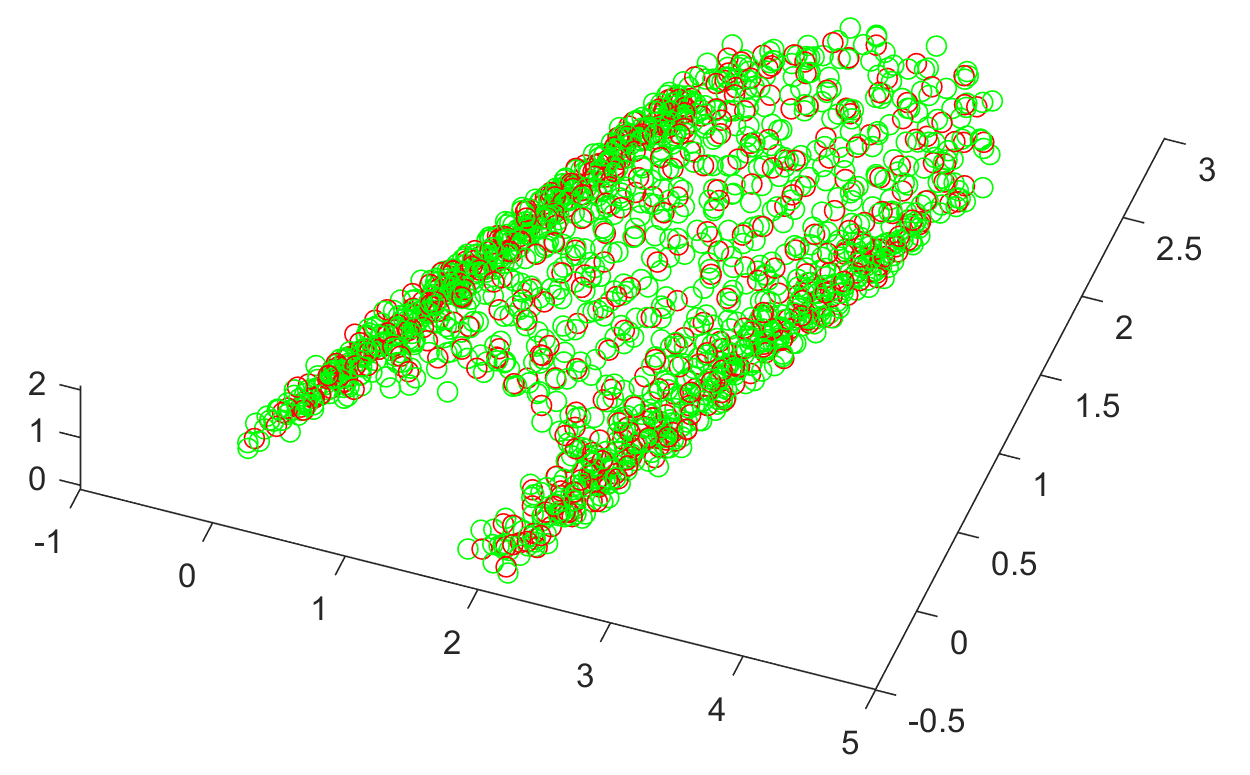} } \\[-0.7ex]
	\subfloat[][]{ \includegraphics[width=0.5\textwidth]{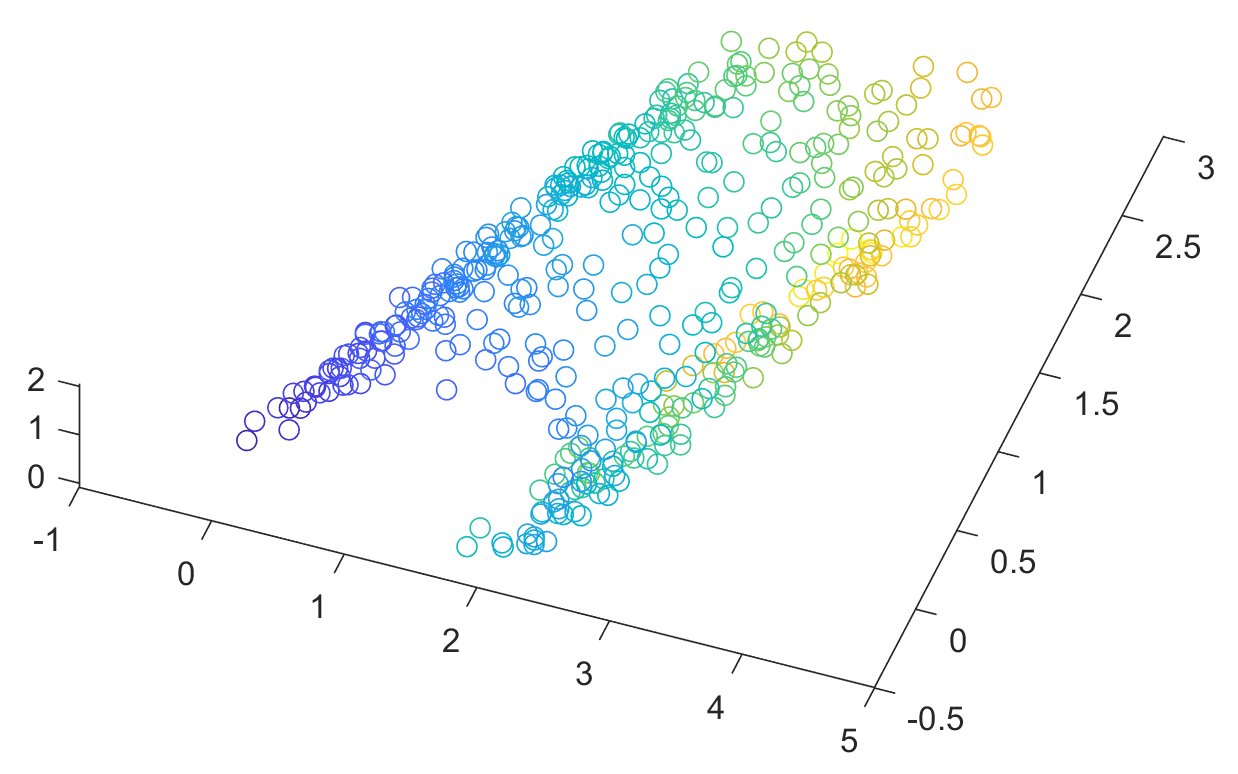} } \hspace{-2em}
	\subfloat[][]{\includegraphics[width=0.5\textwidth]{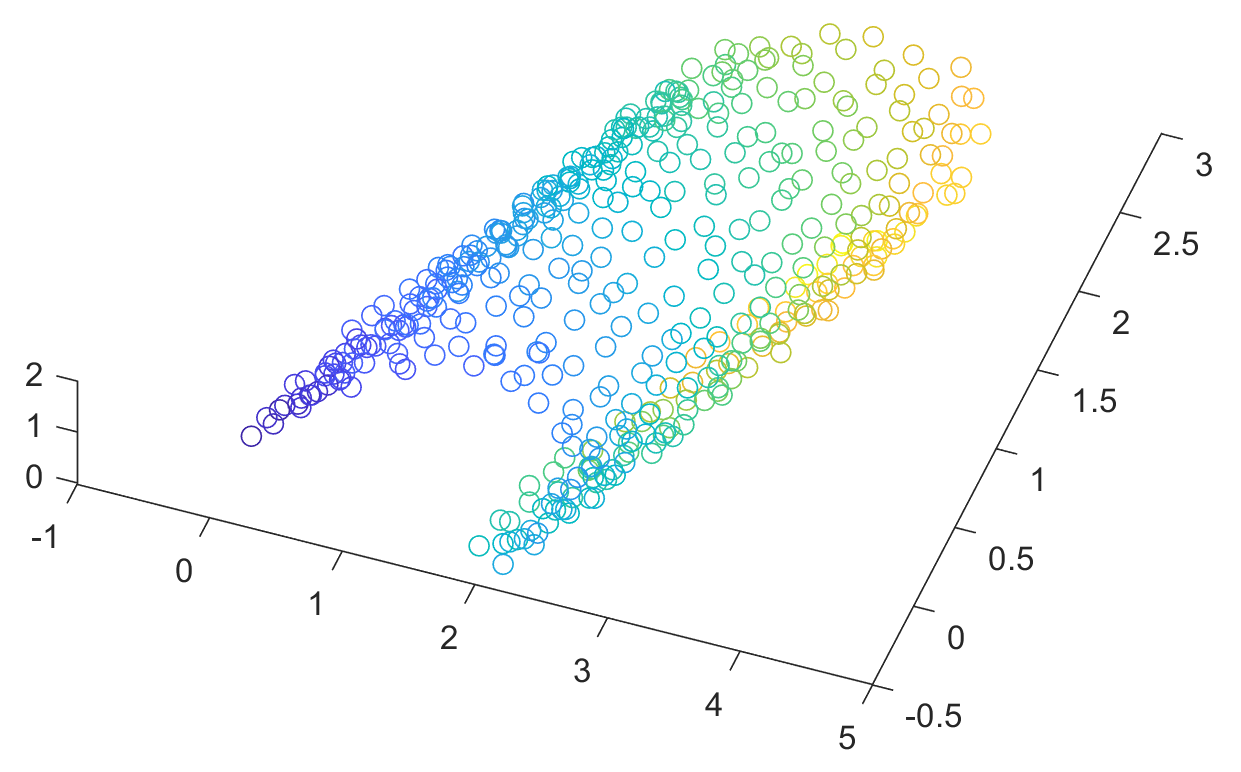}}
	\caption{Six-dimensional cylindrical structure embedded in a 60-dimensional space. Plot of the cross-section of the cylindrical structure in three dimensions. (A) Scattered data with uniformly distributed noise $U(-0.2; 0.2)$ (green), and the initial point-set $Q^{(0)}$ (red). (B) The resulting point-set of the MLOP algorithm after 300 iterations, $Q^{(300)}$ (red) overlaying the noisy samples (green). (C) The initial function values evaluated at the original $Q^{(0)}$ points with noise $U(-0.2, 0.2)$. (D) MLOP approximation at the data points $Q^{(300)}$.}
	\label{fig:func_approx_7Dcylinder}
\end{figure}

%\caption{Summary of the maximum and root mean squared error with standard deviations errors of function approximation on two and six-dimensional cylindrical structure manifold embedded into 60-dimensional space}
\vspace{6mm}
\begin{table}[H]
	\linespread{1.1}\selectfont\centering
	\caption{Summary of the maximum and root mean squared error with standard deviations errors of approximation of a function on a two-dimensional and six-dimensional cylindrical manifold embedded in a 60-dimensional space}
	\begin{tabular}{p{5cm}p{2.3cm}p{2.5cm}p{2.3cm}p{2.5cm}} %{llclc}
		\hline
		& \multicolumn{2}{c|}{2D cylinder in $\R^{60}$}                           & \multicolumn{2}{c}{6D cylinder in $\R^{60}$}                   \\ \cline{2-5} 
		& Max relative error & \multicolumn{1}{c|}{RMSE $\pm$ var}    & Max relative error & RMSE $\pm$ var                \\ \hline
		\textbf{Error over $Q^{(k)}$}                 & \textbf{}                     & \multicolumn{1}{l}{\textbf{}}          & \textbf{}                     & \multicolumn{1}{l}{\textbf{}} \\ \hline
		$f(Q^{(0)})$                                                   & 0.11                          & \multicolumn{1}{c|}{$0.09 \pm 0.0029$} & 0.066                         & $0.08 \pm 0.0018$             \\
		$f(Q^{(300)}) $                                              & 0.06                          & \multicolumn{1}{c|}{$0.05 \pm 0.0012$} & 0.054                         & $0.06 \pm 0.0012$             \\ \hline
		\multicolumn{5}{l}{\textbf{Error over 100 new points}}                                                                                                                               \\ \hline
		RBF with $\phi_1$, centers at $Q^{(0)}$, noisy $f$             & 1.37                          & \multicolumn{1}{c|}{$0.3 \pm 0.04$}    & 0.42                          & $0.4 \pm 0.063$               \\
		RBF with $\phi_1$, centers at $Q^{(300)}$, cleaned $\widetilde f$ & 0.38                          & \multicolumn{1}{c|}{$0.13 \pm 0.007$}  & 0.26                        & $0.26 \pm 0.026$              \\
		RBF with $\phi_2$, centers at $Q^{(300)}$, cleaned $\widetilde f$ & 0.11                          & \multicolumn{1}{c|}{$0.05 \pm 0.0009$} & 0.08                          & $0.07 \pm 0.0018$             \\
		RBF with $\phi_3$, centers at $Q^{(300)}$, cleaned $\widetilde f$ & 0.10                          & \multicolumn{1}{c|}{$0.04 \pm 0.0006$} & 0.13                          & $0.04 \pm 0.0007$             \\
		Weighted average                                           & 0.2                           & \multicolumn{1}{c|}{$0.1 \pm 0.0024$}  & 0.09                          & $0.06 \pm 0.0027$             \\ \hline
	\end{tabular}
	\label{3D_6Dcylinder}
\end{table}

\vspace{-12mm}

\subsection*{Robustness to Noise}

In the following example, we examine the effect of the noise level in the target domain on the quality of the approximation.
To do this numerically, we sampled a function over a Swiss Roll using the parameterization 
\begin{eqnarray*}
	p=\frac{1}{10} [x, y, z, 0, \dotsc, 0]\,,
\end{eqnarray*} 
where $x = t \sin (t)$, $y$ is a random number in the range $[-6, 6]$, and $z = t \cos (t)$, with $t = 8k/n+2$ and $k \in \N$. The approximated function was $f(p)  = t$. We created a Swiss Roll with 800 data points, and randomly sampled 200 points as the initial $Q$-set. We added noise with various  magnitudes (0.1, 0.2, 0.5, and 0.7) to the $P$-points as well as to the values of $f$ at the 
$P$-points. For example, Figure \ref{fig:swiss_roll_func_approx_noise2} left shows a case of approximation with uniformly distributed noise $U(-0.2, 0.2)$, while the right plot presents the denoised version. We see that the data were cleaned both in the domain and in the codomain of the function. In Figure \ref{fig:error_eval_chart_swiss_roll}, we plot the error values under various noise scenarios in the codomain, both for the noisy $Q^{(0)}$ data and the error of the RBF approximation on the $Q^{(300)}$ data. One can see that although the approximation error increases on the noisy data (from 0.03 to 0.22), the approximation error  on clean and quasi-uniformly distributed data error grows only moderately (from 0.02 to 0.18). We also see that at high levels of noise (e.g., 0.7) the accuracy is good. This shows the strengths of our approach, and justifies the need for data denoising as well as uniform sampling before approximation algorithms are applied.

%XXXX remove A-B

\vspace{-12mm}
\begin{figure}[H]
	\centering
	\captionsetup[subfloat]{farskip=0pt,captionskip=0pt, aboveskip=0pt}
	\subfloat[][]{ \includegraphics[width=0.5\textwidth]{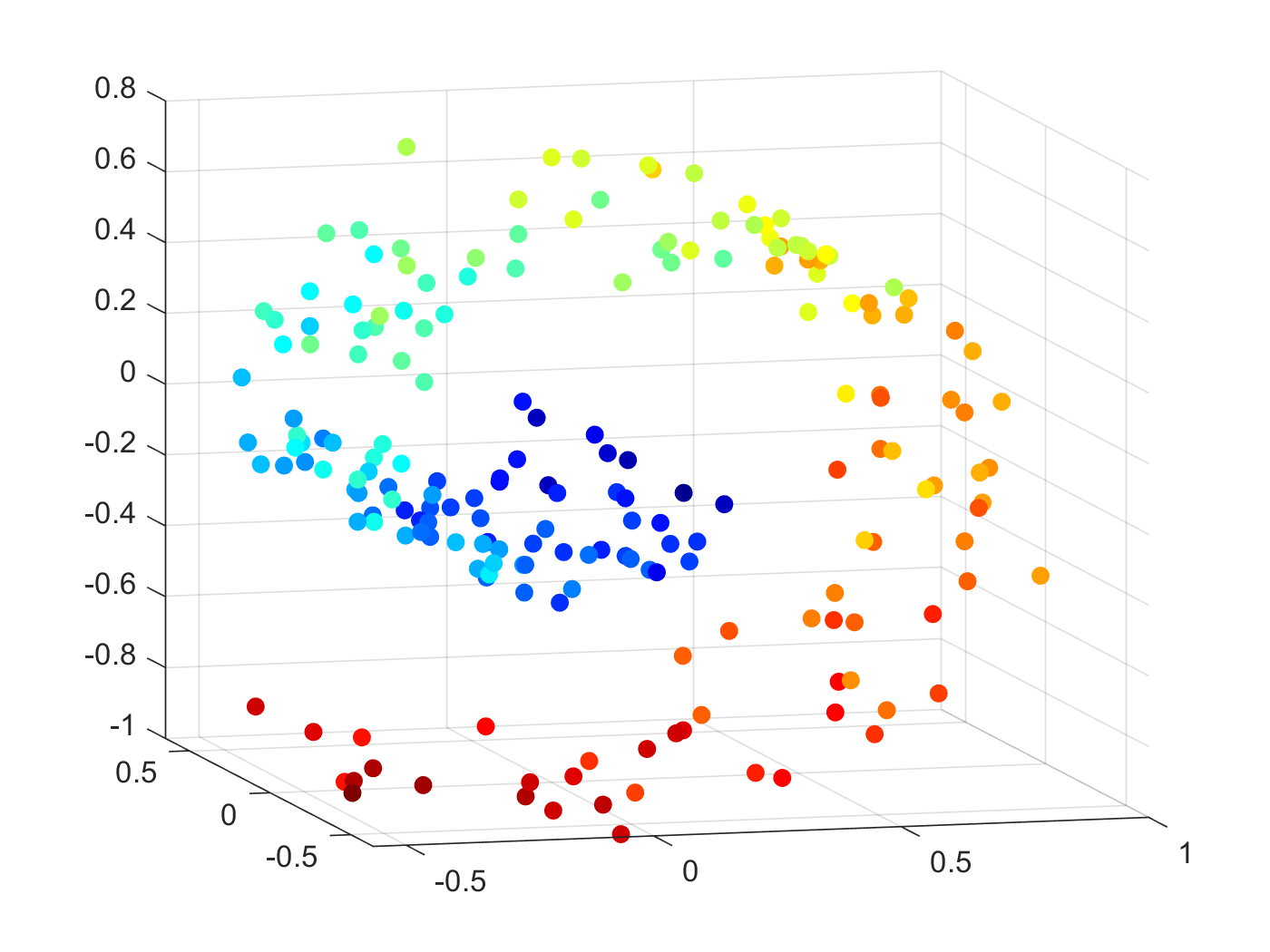} } \hspace{-2em}
	\subfloat[][]{ \includegraphics[width=0.5\textwidth]{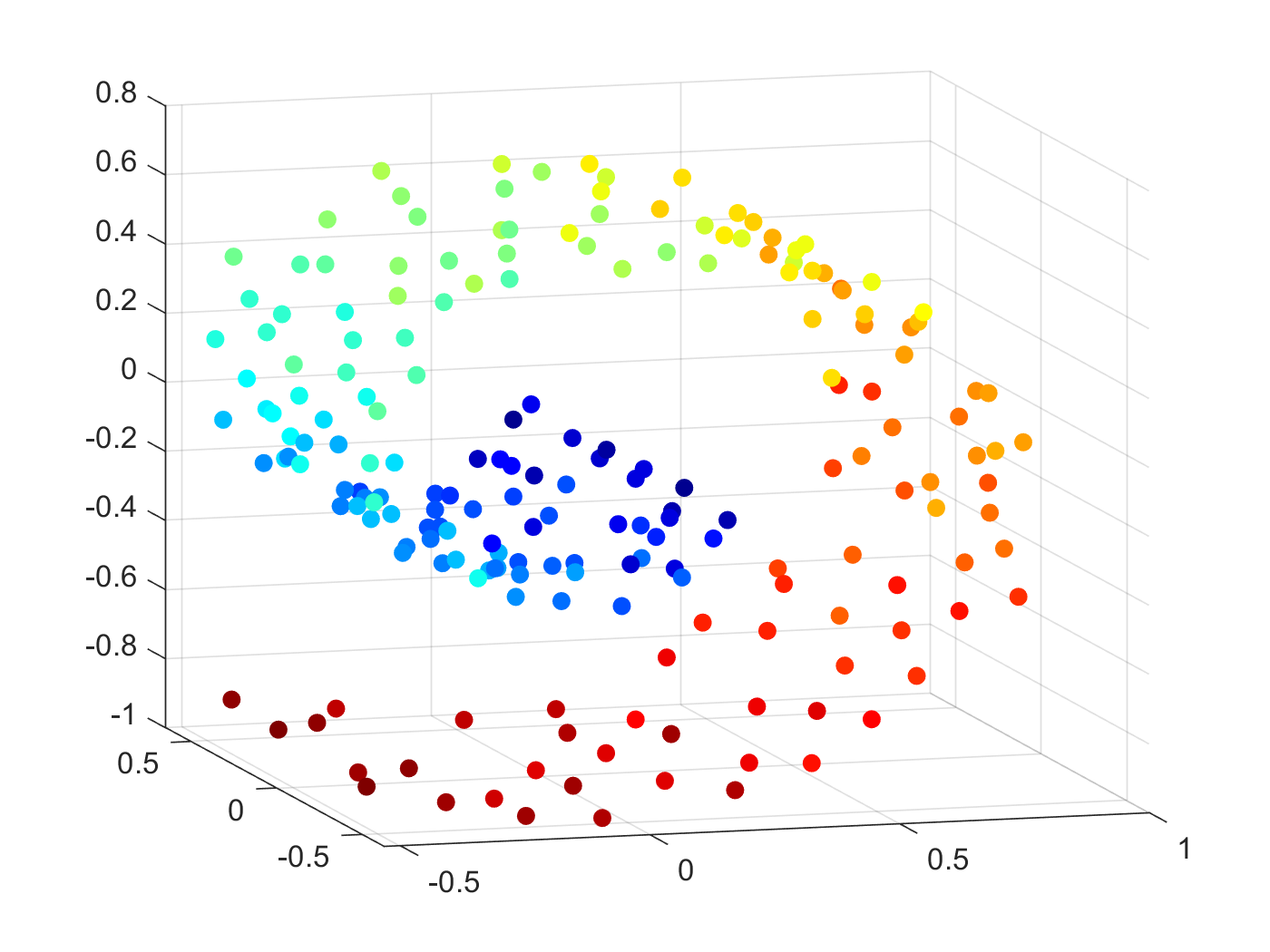} }
	
	\caption{Swiss Roll embedded in $\R^{60}$. The figure depicts the first tree coordinates. Left: The initial values of the function at the original $Q^{(0)}$ points with noise $U(-0.2, 0.2)$, with values indicated by the color. Right: MLOP function approximation at the data points $Q^{(300)}$ cleaned via the MLOP.}
	\label{fig:swiss_roll_func_approx_noise2}
\end{figure}

\begin{figure}[H]
	\centering
	\label{fig:a3}\includegraphics[width=\textwidth,height=6cm,keepaspectratio]{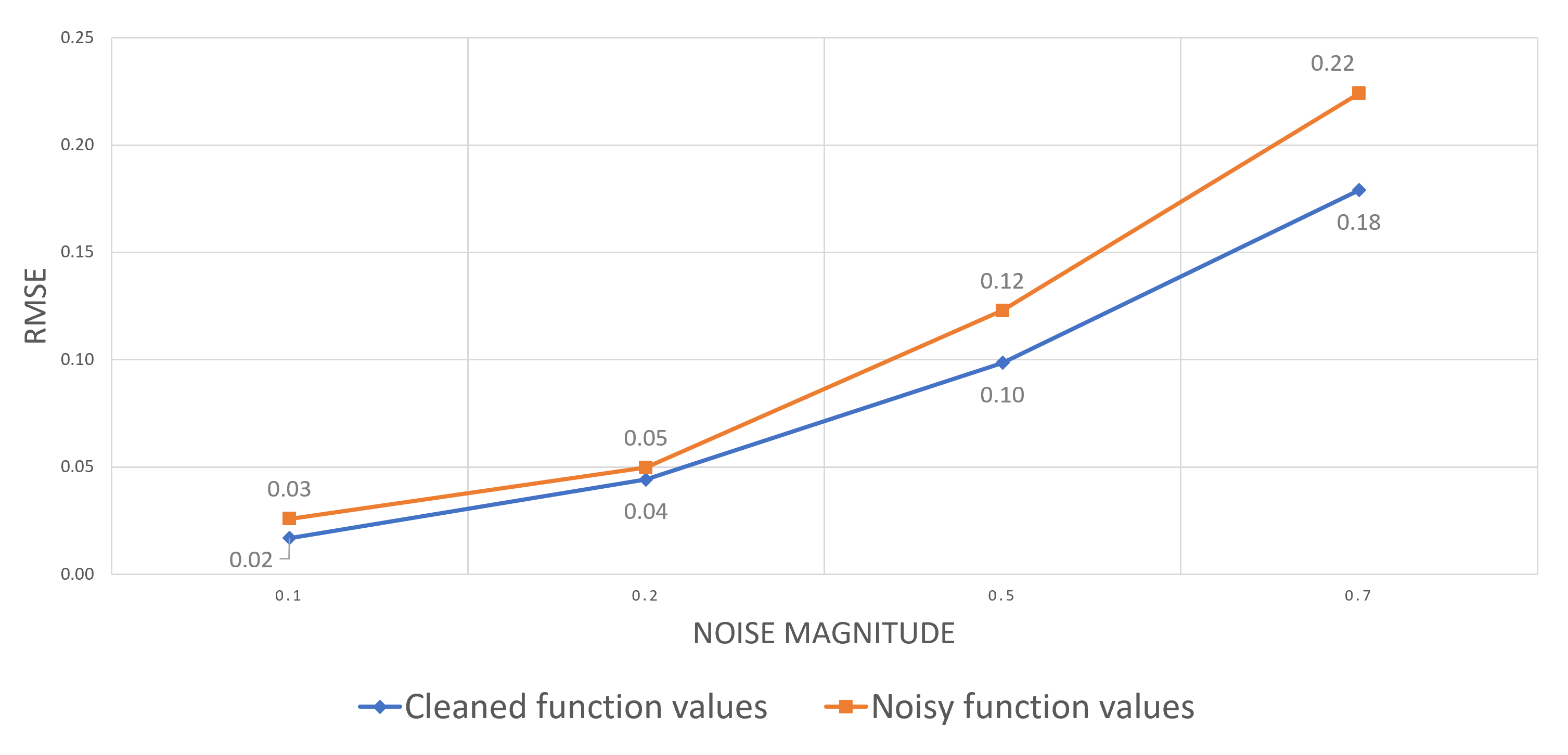}
	\caption{Effect of noise level on the accuracy of function approximation for a Swiss Roll embedded in  a 60-dimensional space. The RMSE error evaluated on the original noisy data is shown in orange, while the RMSE error on the cleaned data is presented in blue.}
	\label{fig:error_eval_chart_swiss_roll}
\end{figure}

\section{Discussion and Conclusions}
\label{sec:conclusions}

In this paper, we consider the problem of approximating a function on a manifold in high dimensions, with noise present both in the domain and in the codomain of the function. Given a set of points  and the values of an unknown function evaluated at these points, the goal is to find the approximation of the function at a new dataset. While in low dimensions this problem did receive a lot of attention in approximation theory, in high dimensions the solution is challenged by the curse of dimensionality. In our solution, we propose to combine the best of both worlds, the Manifold Locally Optimal Projection (MLOP) \cite{faigenbaumgolovin2020manifold} and the method of Radial Basis Functions (RBF) methods \cite{dyn1983iterative}. The approximation problem has two steps. First, find a noise-free representation of the manifold in terms of new points, and the noise-free values of the function at these points. Then, estimate the value of the function at a new given point $x$ using the RBF method,  with the centers set at the cleaned points. The MLOP method is used here for noise removal and for the generation of a quasi-uniform manifold sampling, as a pre-processing step for the RBF mehod; this improves the approximation results dramatically. In the paper, we showed that the order of approximation at the new data points is less than $C_1 h^2 +C_2 h_2^k$, where $h$ is the representative distance of the graph of the function $f$ of the initial data, and $h_2$ is the representative distance of the MLOP reconstruction.

A possible future direction would be to investigate a classification problem. Thus, given data that lies on a manifold, and the corresponding values of a non-smooth function which received only $k$ values, each representing a different class. In addition, both the data and the labeling contain noise and outliers. The research question to be addressed is whether new data can be classified with high accuracy.

An additional possible future direction would be to deal with optimization of functions  on a manifold. In the past decade, optimization gained a lot of attention, especially with the rise of Neural Networks computing (NN). Optimization algorithms are the pillar stones of the NNs, as they are in charge of constructing the networks, by learning from the training examples. In our research, we propose introducing a new optimization process, that will take into account the topology of the data. We would like to utilize the manifold structure of the data to improve the optimization process, by incorporating the manifold's  information into the NN optimization of a function. We propose extending the MLOP framework to deal with this task, by modifying the definition of the MLOP cost function to include the optimized function. By extending the definition of the MLOP algorithm, the gradient descent iterations will find not only the optimal manifold reconstruction, but also minimize the function.

\section*{Acknowledgments}
We would like to thank Dr. Barak Sober for valuable discussions  and comments. This study was supported by a generous donation from Mr. Jacques Chahine, made through the French Friends of Tel Aviv University, and was partially supported by ISF grant 2062/18.

\bibliographystyle{spmpsci}      % basic style, author-year citations
\bibliography{references_MLOP}
%\bibliographystyle{spmpsci}  spbasic     % mathematics and physical sciences

%\begin{thebibliography}{}
% etc
%\end{thebibliography}

\end{document}